\documentclass[12pt,english]{amsart}
\usepackage[cp1251]{inputenc}
\usepackage[english]{babel}
\usepackage{amsmath,amsthm,amssymb}
\newtheorem{theorem}{Theorem}[section]
\newtheorem{lemma}[theorem]{Lemma}
\newtheorem{corollary}[theorem]{Corollary}

\newtheorem{conjecture}[theorem]{Conjecture}
\numberwithin{equation}{section}

\theoremstyle{definition}
\newtheorem{definition}[theorem]{Definition}
\newtheorem{remark}[theorem]{Remark}

\begin{document}
	
\title[]{
Dimension-free Markov--Bernstein inequalities 
for product measures
}
	
\author{Egor Kosov}

\address{\noindent Egor Kosov,
Centre de Recerca Matem\`atica, Campus de Bellaterra, Edifici~C 08193
Bellaterra (Barcelona), Spain.}
\email{kosoved09@gmail.com}

\subjclass[2020]{Primary 41A17; Secondary 41A63, 42C10, 60E15, 28C20}

\keywords{Markov--Bernstein inequality; Gaussian measure; product measure; Freud weight; unimodal density}

\begin{abstract}
For even integer exponents $p$, we obtain dimension-free
Markov--Bernstein inequalities for polynomials with respect to product
probability measures $\mu=\mu_1\otimes\cdots\otimes\mu_n$.
When the measures $\mu_j$ have unimodal densities, we prove
\[
\|\nabla f\|_{L^p(\mu)}
\le
C(p) d^2
\|f\|_{L^p(\mu)}
\]
for polynomials $f$ of degree at most $d$.
The dependence on $d$ is optimal already for the uniform distribution on the
unit cube. For products of one-dimensional Freud measures with densities
proportional to $e^{-|t|^{2m}}$, the factor $d^2$ can be replaced by
$d^{1-\frac{1}{2m}}$.

In the Gaussian case, for all $p\ge4$, we prove that
\[
\|\nabla f\|_{L^p(\gamma^n)}
\le
C(p)d^{\frac12+\theta_p}\|f\|_{L^p(\gamma^n)}
\]
for every polynomial $f$ of degree at most $d$, where
$\theta_p\le \frac{2}{3p}$ and $\theta_p=0$ whenever $p$ is an even integer.
Thus, in the even-integer case, we establish the sharp dependence on the
degree conjectured by Eskenazis--Ivanisvili. For general $p\ge4$, the
estimate improves upon their dimension-free inequality.
\end{abstract}

\maketitle

\section{Introduction}

Markov--Bernstein inequalities estimate norms of derivatives of algebraic
polynomials in terms of norms of the polynomials themselves. In the
one-dimensional case, the classical Markov inequality on an interval asserts
that
\[
\|f'\|_{L^\infty([-1,1])}
\le
d^2\|f\|_{L^\infty([-1,1])}
\]
for every algebraic polynomial $f$ of degree at most $d$. Bernstein's
inequality refines this estimate by introducing a weight which compensates for
the endpoint effect:
\[
\|(1-x^2)^{1/2}f'\|_{L^\infty([-1,1])}
\le
d\|f\|_{L^\infty([-1,1])}.
\]
The factors $d^2$ and $d$ in these inequalities are sharp, as can already be
seen from the Chebyshev polynomials. We refer to
\cite{BorweinErdelyi,KNT21} for detailed discussions of classical Markov and
Bernstein inequalities. Integral versions of such inequalities were studied,
for instance, in \cite{BE95, HST37}, while weighted variants were
developed in \cite{Fr71,Fr77,LeLu94,MT}. Besides their intrinsic
interest, these inequalities play an important role in many areas of analysis,
especially in approximation theory, see, for example, \cite[Chapter~4]{DL-book} and \cite[Chapter~8]{DT}.

Multivariate analogues of the classical Markov--Bernstein inequalities, with
the derivative replaced by the Euclidean norm of the gradient, have also been
extensively studied in both the uniform norm and integral $L^p$ norms; see, for
instance,
\cite{Kellogg28,KR99,Pierzchala16,Skalyga01,Wilhelmsen74}
and \cite{Dai06,IO22,Kr09,Kr20,Kr21,Kr24}, respectively.
A recent line of work of Kro\'o, Dai, and Prymak
\cite{DaiPrymak22,DaiPrymakMeshes24,DKP25,Kr11,Kr26}
demonstrates the importance of multivariate Markov--Bernstein inequalities in
the construction of optimal polynomial meshes and in Marcinkiewicz-type
discretization of $L^p$ norms.

Although there is a separate line of work on Markov--Bernstein inequalities in the uniform norm in Banach-space settings (see, for instance, \cite{Harr, RevSar, Sarant91, Skalyga97}), most of the integral Markov--Bernstein inequalities cited above belong to the classical finite-dimensional framework, in which the ambient dimension and the underlying domain are fixed and the constants may depend on both the dimension and the geometry of the domain.
In many modern problems at the interface of high-dimensional analysis and
probability, a different point of view is more natural. The ambient dimension
may be arbitrarily large, and one seeks estimates with constants independent
of the dimension. In the Gaussian setting, such dimension-free results for
polynomials are especially natural, since they can be transferred to abstract
Wiener spaces and, in particular, to finite sums of multiple stochastic
integrals, see \cite{Nualart06}.

Dimension-free estimates also arise naturally in questions concerning
anti-concentration of polynomials \cite{CarWr}, regularity of distributions of
polynomials in log-concave random vectors \cite{Kos18}, total variation
distance estimates between such distributions \cite{Kos21}, and oscillatory
integrals with polynomial phases \cite{Kos25}. In these problems, one often
needs to control the non-degeneracy of polynomial distributions, which is
typically measured in terms of the variance of the polynomial with respect to
the underlying measure. Markov--Bernstein-type inequalities are well suited
for this purpose, since, when applied iteratively, they relate the leading
coefficients of a polynomial to its $L^2$ norm, and hence to its variance after
centering. Related dimension-free coefficient estimates were recently studied
by Glazer--Mikulincer \cite{GM22,GM26}.

From this dimension-free viewpoint, integral Markov--Bernstein inequalities
for the full gradient seem to have been investigated much less systematically
and, to the best of our knowledge, are known only in a few specific settings.
Namely, Kro\'o--Szabados \cite{KrSz12} obtained $L^2$ Markov--Bernstein
inequalities for products of Gaussian and Gamma distributions, while general
$L^p$ Gaussian Markov--Bernstein inequalities were studied by
Eskenazis--Ivanisvili~\cite{EI}.

The main goal of this paper is to study dimension-free integral $L^p$
Markov--Bernstein inequalities with sharp dependence on the degree $d$.
For general high-dimensional measures, such estimates are not available for
the whole class of algebraic polynomials, even in the log-concave setting, as
already shown by the normalized Lebesgue measure on the isotropic Euclidean
ball~\cite{GM22}. 
We therefore focus on product measures and prove sharp dimension-free estimates
for even integer values of $p$ for products of measures with bounded unimodal
densities and for products of measures with certain Freud-type densities. For
Gaussian measures, in particular, our approach confirms the
Eskenazis--Ivanisvili conjecture for even integer exponents and also yields
improved estimates for all $p\ge4$.

\subsection{Gaussian Markov--Bernstein inequalities}
Let $\gamma$ denote the standard Gaussian measure on $\mathbb R$. Classical
results of Freud \cite{Fr71,Fr77} imply that, for $p\in[1,\infty)$,
\[
\|f'\|_{L^p(\gamma)}
\le
C(p)\sqrt d\,\|f\|_{L^p(\gamma)}
\]
for every polynomial $f$ of degree at most $d$.

Motivated by this one-dimensional estimate, Eskenazis and Ivanisvili
\cite{EI} formulated the following multidimensional conjecture.

\begin{conjecture}\label{conj}
For every $p\in[1,\infty)$, there exists a constant $C(p)>0$ such that, for every
$n,d\in\mathbb N$, one has
\[
\|\nabla f\|_{L^p(\gamma^n)}
\le
C(p)\sqrt d\,\|f\|_{L^p(\gamma^n)}
\quad \forall f\in \mathcal P_d(\mathbb R^n).
\]
\end{conjecture}
Here $\gamma^n$ denotes the $n$-fold product of the one-dimensional measure
$\gamma$, and $\mathcal P_d(\mathbb R^n)$ denotes the space of all algebraic
polynomials of degree at most $d$ in $n$ variables.

For $p=2$, the standard Hermite
polynomial expansion implies
\begin{equation}\label{eq-MB-2}
\|\nabla f\|_{L^2(\gamma^n)}
\le
\sqrt d\,\|f\|_{L^2(\gamma^n)}
\quad \forall f\in \mathcal P_d(\mathbb R^n).
\end{equation}

For $p\in[1,2)\cup(2,\infty)$, Eskenazis and Ivanisvili \cite{EI} proved the
following weaker estimate:
\begin{equation}\label{eq-EI-est}
\|\nabla f\|_{L^p(\gamma^n)}
\le
C(p)d^{
\frac12+\frac1\pi\arctan\bigl(\frac{|p-2|}{2\sqrt{p-1}}\bigr)}
\|f\|_{L^p(\gamma^n)}
\quad \forall f\in\mathcal P_d(\mathbb R^n).
\end{equation}

Our first result improves upon this estimate for $p\ge4$ and confirms
Conjecture~\ref{conj} for even integer exponents.

\begin{theorem}\label{th-1}
For every $p\ge4$, there is a constant $C(p)>0$ such that, for every
$n,d\in\mathbb N$ and every $f\in\mathcal P_d(\mathbb R^n)$, one has
\[
\|\nabla f\|_{L^p(\gamma^n)}
\le
C(p)d^{\frac12+\theta_p}\|f\|_{L^p(\gamma^n)},
\]
where
\[
\theta_p:=
\frac1\pi\arctan\biggl(
\frac{p-2\lfloor p/2\rfloor}
{2\sqrt{\lfloor p/2\rfloor(p-\lfloor p/2\rfloor)}}
\biggr)\le \frac{2}{3p}.
\]
In particular, if $p$ is an even integer, then $\theta_p=0$ and one may take
$C(p)=3\sqrt p$.
\end{theorem}

The proof is carried out in two main steps. First, an integration by parts
argument gives
\[
\|\nabla f\|_{L^p(\gamma^n)}^p
\le
C(p)\|f|\nabla f|^{p/2-1}\|_{L^2(\gamma^n)}
\Bigl(
\|\nabla f\|_{L^p(\gamma^n)}^p
+
\int_{\mathbb R^n}
|\nabla f|^{p-2}\|D^2f\|_{\rm HS}^2
\,d\gamma^n
\Bigr)^{1/2},
\]
see Corollary~\ref{cor-1-1}. The last integral should be thought of as a
multidimensional analogue of the one-dimensional expression
$|g|^{p-2}|g'|^2$. When $p=2k$ is an even integer, this expression can be
written as $k^{-2}|(g^k)'|^2$ and can therefore be controlled by the
$L^2$ Markov--Bernstein inequality \eqref{eq-MB-2}, since $g^k$ is again a
polynomial. This reduction is formalized in Lemma~\ref{lem-1-2}. Applying
H\"older's inequality, we then obtain
\[
\|\nabla f\|_{L^{2k}(\gamma^n)}^{2k}
\le
C(k)d^{1/2}
\|f\|_{L^{2k}(\gamma^n)}
\|\nabla f\|_{L^{2k}(\gamma^n)}^{2k-1},
\]
which gives the sharp bound for even integer exponents. To obtain the
improved estimate for all $p\ge4$, we use the Eskenazis--Ivanisvili estimate
\eqref{eq-EI-est} with exponent $p/k$, $k:=\lfloor p/2\rfloor$, in place of the $L^2$ estimate
\eqref{eq-MB-2}.
This approach, with some technical adjustments, is then applied to the unit
cube and to Freud-type densities.

\begin{remark}
While it is natural to try to interpolate between the sharp estimates obtained
for arbitrarily large even values of $p$, it is not clear how to carry out
such an interpolation. A first natural attempt is to consider the linear operator
\[
T_dg:=\nabla P_dg,
\]
where $P_d$ is the $L^2(\gamma^n)$-orthogonal projection onto
$\mathcal P_d(\mathbb R^n)$. This naive approach does not seem to give the
desired result. Indeed, the sharp estimate for polynomials in
$L^{2k}(\gamma^n)$ gives only
\[
\|T_dg\|_{L^{2k}(\gamma^n)}
\le
C(k)\sqrt d\,\|P_dg\|_{L^{2k}(\gamma^n)}.
\]
Thus, interpolation through this operator would require a bound for $P_d$ on
$L^{2k}(\gamma^n)$ that is uniform in $d$. Such a bound is not available for $k\ne1$ already
in dimension one, since Hermite polynomial expansions are known to converge in
$L^p(\gamma)$ only in the case $p=2$, see \cite[Section~10]{Pollard48}.
\end{remark}

\subsection{Uniform distribution on the unit cube}

Sharp dimension-free Markov--Bernstein inequalities on the unit cube
$[-1,1]^n$ in the uniform norm were obtained by
Skalyga~\cite{Skalyga96,Skalyga97}. However, the case of integral norms in
this classical multivariate setting appears to be largely unexplored.
To the best of our knowledge, even the $L^2([-1,1]^n)$ inequality for the full
gradient, with a dimension-free constant and sharp dependence on $d$, has not
been explicitly stated in the literature. We close this gap for even integer
exponents $p$ by proving the following Markov-type and weighted
Bernstein-type inequalities.

\begin{theorem}\label{th-2}
Let $k,n,d\in\mathbb N$. Then, for every
$f\in\mathcal P_d(\mathbb R^n)$, one has
\[
\Bigl(
\int_{[-1,1]^n}
\Bigl(
\sum_{j=1}^n(1-x_j^2)(\partial_{x_j}f)^2
\Bigr)^k\,dx
\Bigr)^{\frac{1}{2k}}
\le
C_1(k)d\,\|f\|_{L^{2k}([-1,1]^n)}
\]
and
\[
\|\nabla f\|_{L^{2k}([-1,1]^n)}
\le
C_2(k)d^2\,\|f\|_{L^{2k}([-1,1]^n)}.
\]
One may take $C_1(1)=\sqrt{2}$, $C_2(1)=12$, and
$C_1(k)=46k^2$, $C_2(k)=368k^3$ for $k\ge2$.
\end{theorem}

The dependence on $d$ is sharp already in dimension one.

\subsection{Unimodal densities}

Next, we transfer the cube estimate to products of probability measures with
bounded unimodal densities.

\begin{definition}
We say that a density $\varrho$ on $\mathbb R$ is unimodal if there exists
$x_0\in\mathbb R$ such that $\varrho$ is nondecreasing on $(-\infty,x_0]$ and
nonincreasing on $[x_0,\infty)$.
\end{definition}

In particular, all superlevel sets $\{\varrho\ge t\}$ are intervals. We use
this fact to represent a unimodal density as a mixture of uniform distributions
on intervals. 
This reduces the Markov--Bernstein inequalities for products of unimodal
densities to the case of products of intervals and gives the following
extension of the cube result.

\begin{theorem}\label{th-4}
Let $k,n,d\in\mathbb N$ and let
$\mu=\mu_1\otimes\cdots\otimes\mu_n$, where each $\mu_j$ has a bounded
unimodal density $\varrho_j$ on $\mathbb R$. Set
\[
M:=\max_{1\le j\le n}\|\varrho_j\|_\infty.
\]
Then, for every $f\in\mathcal P_d(\mathbb R^n)$, one has
\[
\|\nabla f\|_{L^{2k}(\mu)}
\le
C(k)Md^2
\|f\|_{L^{2k}(\mu)}.
\]
One may take $C(1)=24$ and $C(k)=736k^3$ for $k\ge2$.
\end{theorem}

The dependence on $d$ is sharp already for the cube.

Since log-concave densities are unimodal and an isotropic log-concave density
on $\mathbb R$ is bounded by $1$, see \cite[Lemma 5.5(a)]{LV07}, we obtain the
following corollary.

\begin{corollary}\label{MB-log-conc}
Let $k,n,d\in\mathbb N$ and let
$\mu=\mu_1\otimes\cdots\otimes\mu_n$, where $\mu_1,\ldots,\mu_n$ are isotropic
log-concave probability measures on $\mathbb R$. Then, for every
$f\in\mathcal P_d(\mathbb R^n)$, one has
\[
\|\nabla f\|_{L^{2k}(\mu)}
\le
C(k)d^2
\|f\|_{L^{2k}(\mu)}.
\]
One may take $C(1)=24$ and $C(k)=736k^3$ for $k\ge2$.
\end{corollary}

\subsection{Freud-type densities}

The dependence on $d$ in Theorem~\ref{th-4} and
Corollary~\ref{MB-log-conc} is sharp for the whole classes of bounded unimodal
densities and isotropic log-concave densities, respectively. For individual
densities, however, this dependence may be improved, as illustrated by the
Gaussian case.

In the last part of the paper, we extend the Gaussian estimates for even
integer exponents $p$ to products of probability measures with Freud-type densities,
namely
\[
\nu_m(dt)=c_m e^{-|t|^{2m}}\,dt,
\]
where $m\in\mathbb N$ and $c_m$ is the normalizing constant. In dimension one,
Markov--Bernstein inequalities for general Freud weights were obtained by
Levin and Lubinsky~\cite{LeLu94}. In the particular case of the measures
$\nu_m$, for arbitrary $p\in[1,\infty)$, this estimate asserts, see
\cite[Estimate~(7.3)]{LuSurvey07}, that
\begin{equation}\label{Freud-MB}
\|f'\|_{L^p(\nu_m)}
\le
C(m,p)d^{1-\frac1{2m}}\|f\|_{L^p(\nu_m)}
\quad \forall f\in \mathcal P_d(\mathbb R).
\end{equation}
For even integer $p$, we extend this result to the multidimensional product
setting.

\begin{theorem}\label{th-6}
Let $m,k\in\mathbb N$. There exists a constant $C(m,k)>0$ such that, for
every $n,d\in\mathbb N$ and every $f\in\mathcal P_d(\mathbb R^n)$, one has
\[
\|\nabla f\|_{L^{2k}(\nu_m^n)}
\le
C(m,k)d^{1-\frac1{2m}}
\|f\|_{L^{2k}(\nu_m^n)}.
\]
\end{theorem}

\subsection{Notation}

As mentioned above, $\mathcal P_d(\mathbb R^n)$ denotes the space of all
algebraic polynomials on $\mathbb R^n$ of degree at most $d$. We denote by
$C^\infty(\mathbb R^n)$ the space of all smooth functions on $\mathbb R^n$,
and by $C^\infty(\mathbb R^n,\mathbb R^n)$ the space of all smooth vector
fields on $\mathbb R^n$.
We denote by $C^\infty_P(\mathbb R^n)$ the space of all smooth functions on
$\mathbb R^n$ such that the function itself and all its partial derivatives
have at most polynomial growth. Similarly,
$C^\infty_P(\mathbb R^n,\mathbb R^n)$ denotes the space of all vector fields
whose components belong to $C^\infty_P(\mathbb R^n)$.

Throughout the paper, constants are denoted by $C$ and may change from line to
line. Their dependence on parameters is always indicated explicitly, for
instance by writing $C(p)$, $C(k)$, or $C(m,k)$. When several constants appear
in the same argument, we distinguish them by subscripts, writing
$C_1,C_2$, etc. These subscripts are only labels and do not indicate any
additional dependence.

\section{The Gaussian case}	

\subsection{Gaussian integration by parts}

For $f\in C^\infty_P(\mathbb R^n)$, set
\[
Lf:=\Delta f-\langle x,\nabla f\rangle .
\]
For $j\in\{1,\ldots,n\}$, define
\[
\delta_{x_j}f:=x_jf-\partial_{x_j}f .
\]
For a vector field $u\in C^\infty_P(\mathbb R^n,\mathbb R^n)$, define
\[
\delta u:=\sum_{j=1}^n \delta_{x_j}u_j
=\langle x,u\rangle-\operatorname{div}u .
\]
In particular,
\[
\delta\nabla f=-Lf.
\]

For $f,g\in C^\infty_P(\mathbb R^n)$, the Gaussian integration by parts
formula gives
\[
\int_{\mathbb R^n}\partial_{x_j}f \, g\,d\gamma^n
=
\int_{\mathbb R^n} f\,\delta_{x_j}g\,d\gamma^n .
\]
Consequently, for $f\in C^\infty_P(\mathbb R^n)$ and
$u\in C^\infty_P(\mathbb R^n,\mathbb R^n)$,
\begin{equation}\label{int-by-parts-1}
\int_{\mathbb R^n}\langle \nabla f,u\rangle\,d\gamma^n
=
\int_{\mathbb R^n} f\,\delta u\,d\gamma^n .
\end{equation}

We will also use the following consequence. Since
\[
\partial_{x_j}(\delta u)
=
u_j+\sum_{k=1}^n\delta_{x_k}\partial_{x_j}u_k ,
\]
we obtain
\begin{align}\label{int-by-parts-2}
\int_{\mathbb R^n}(\delta u)^2\,d\gamma^n
&=
\int_{\mathbb R^n}
\sum_{j=1}^n u_j
\Bigl(
u_j+\sum_{k=1}^n\delta_{x_k}\partial_{x_j}u_k
\Bigr)\,d\gamma^n
\\&=
\int_{\mathbb R^n}|u|^2\,d\gamma^n
+
\int_{\mathbb R^n}
\sum_{j,k=1}^n
\partial_{x_k}u_j\,\partial_{x_j}u_k\,d\gamma^n
\notag
\\
&\le
\int_{\mathbb R^n}|u|^2\,d\gamma^n
+
\int_{\mathbb R^n}
\sum_{j,k=1}^n
(\partial_{x_k}u_j)^2\,d\gamma^n .
\notag
\end{align}

\subsection{Key lemmas}

\begin{lemma}\label{lem-1-1}
Let $n,d\in\mathbb N$, let $r\ge q>0$, and let
$f\in\mathcal P_d(\mathbb R^n)$. Then
\begin{align*}
\|\nabla f\|_{L^{r+2}(\gamma^n)}^{r+2}
&\le
2\|f|\nabla f|^{r-q}\|_{L^2(\gamma^n)}
\Bigl(
\int_{\mathbb R^n}
|\nabla f|^{2q+2}
+
|\nabla f|^{2q}\|D^2f\|_{\rm HS}^2
\\
&\qquad\qquad\qquad\qquad\qquad\qquad\qquad\quad+
r^2|\nabla f|^{2q-2}|D^2f\nabla f|^2\,d\gamma^n
\Bigr)^{1/2}.
\end{align*}
\end{lemma}

\begin{proof}
Let $\eta\in C^\infty(\mathbb R)$ be a function such that
\[
\eta(t)=0 \quad \text{for } |t|\le 1,
\quad
\eta(t)=1 \quad \text{for } |t|\ge 2,
\quad
0\le \eta(t)\le 1 \quad \text{for all } t\in\mathbb R.
\]
For $\varepsilon>0$, set
$\eta_\varepsilon(t):=\eta(t/\varepsilon)$.

For $a>0$, consider the vector field
\[
u_{a,\varepsilon}
:=
|\nabla f|^a\eta_\varepsilon(|\nabla f|)\nabla f,
\]
where the expression is understood as zero on the set
$\{|\nabla f|\le\varepsilon\}$. 
With this convention, $u_{a,\varepsilon}\in C^\infty_P(\mathbb R^n,\mathbb R^n)$.

We have
\[
\delta u_{a,\varepsilon}
=
-|\nabla f|^a\eta_\varepsilon(|\nabla f|)Lf
-
\bigl(
a|\nabla f|^{a-2}\eta_\varepsilon(|\nabla f|)
+
|\nabla f|^{a-1}\eta_\varepsilon'(|\nabla f|)
\bigr)
\langle D^2f\nabla f,\nabla f\rangle .
\]
Therefore,
\begin{equation}\label{eq-delta-u}
\delta u_{r,\varepsilon}
=
|\nabla f|^{r-q}
\bigl(
\delta u_{q,\varepsilon}
-
(r-q)|\nabla f|^{q-2}\eta_\varepsilon(|\nabla f|)
\langle D^2f\nabla f,\nabla f\rangle
\bigr).
\end{equation}

By Gaussian integration by parts \eqref{int-by-parts-1},
\[
\int_{\mathbb R^n}
|\nabla f|^{r+2}I_{\{|\nabla f|\ge 2\varepsilon\}}\,d\gamma^n
\le 
\int_{\mathbb R^n}
|\nabla f|^{r+2}\eta_\varepsilon(|\nabla f|)\,d\gamma^n
=
\int_{\mathbb R^n}
\langle \nabla f,u_{r,\varepsilon}\rangle\,d\gamma^n
=
\int_{\mathbb R^n} f\,\delta u_{r,\varepsilon}\,d\gamma^n .
\]
Using the identity \eqref{eq-delta-u} and applying the
Cauchy--Schwarz inequality, we obtain
\begin{align*}
&\int_{\mathbb R^n}
|\nabla f|^{r+2}I_{\{|\nabla f|\ge 2\varepsilon\}}\,d\gamma^n
\\
&\le
\Bigl(
\int_{\mathbb R^n}
|f|^2|\nabla f|^{2r-2q}\,d\gamma^n
\Bigr)^{1/2}
\Bigl(
\int_{\mathbb R^n}
\bigl(
\delta u_{q,\varepsilon}
-
(r-q)|\nabla f|^{q-2}\eta_\varepsilon(|\nabla f|)
\langle D^2f\nabla f,\nabla f\rangle
\bigr)^2\,d\gamma^n
\Bigr)^{1/2}.
\end{align*}

We estimate the second factor. Since
\[
\langle D^2f\nabla f,\nabla f\rangle^2
\le
|\nabla f|^2|D^2f\nabla f|^2,
\]
we get
\begin{align*}
&\int_{\mathbb R^n}
\Bigl(
\delta u_{q,\varepsilon}
-
(r-q)|\nabla f|^{q-2}\eta_\varepsilon(|\nabla f|)
\langle D^2f\nabla f,\nabla f\rangle
\Bigr)^2\,d\gamma^n
\\
&\le
2\int_{\mathbb R^n}(\delta u_{q,\varepsilon})^2\,d\gamma^n
+
2(r-q)^2
\int_{\mathbb R^n}
|\nabla f|^{2q-2}\eta_\varepsilon^2(|\nabla f|)
|D^2f\nabla f|^2\,d\gamma^n .
\end{align*}
By \eqref{int-by-parts-2}, applied to $u_{q,\varepsilon}$,
\[
\int_{\mathbb R^n}(\delta u_{q,\varepsilon})^2\,d\gamma^n
\le
\int_{\mathbb R^n}|u_{q,\varepsilon}|^2\,d\gamma^n
+
\int_{\mathbb R^n}
\sum_{j,k=1}^n
\bigl(\partial_{x_k}(u_{q,\varepsilon})_j\bigr)^2\,d\gamma^n .
\]
We have
\[
|u_{q,\varepsilon}|^2
=
|\nabla f|^{2q+2}\eta_\varepsilon^2(|\nabla f|)
\le
|\nabla f|^{2q+2}
\]
and
\[
\partial_{x_k}(u_{q,\varepsilon})_j
=
|\nabla f|^q\eta_\varepsilon(|\nabla f|)
\partial_{x_kx_j}^2f
+
\bigl(
q|\nabla f|^{q-2}\eta_\varepsilon(|\nabla f|)
+
|\nabla f|^{q-1}\eta_\varepsilon'(|\nabla f|)
\bigr)
\langle \nabla\partial_{x_k}f,\nabla f\rangle
\partial_{x_j}f .
\]
Therefore,
\[
\sum_{j,k=1}^n
\bigl(\partial_{x_k}(u_{q,\varepsilon})_j\bigr)^2
\le
2|\nabla f|^{2q}\|D^2f\|_{\rm HS}^2
+
2\bigl(
q|\nabla f|^{q-1}
+
|\nabla f|^q|\eta_\varepsilon'(|\nabla f|)|
\bigr)^2
|D^2f\nabla f|^2 .
\]
Since
$\eta_\varepsilon'(t)=\varepsilon^{-1}\eta'(t/\varepsilon)$
and $\eta_\varepsilon'$ is supported on
$\{\varepsilon\le |t|\le 2\varepsilon\}$, we have
\[
|\nabla f|^q|\eta_\varepsilon'(|\nabla f|)|
\le
\frac{|\nabla f|^q}{\varepsilon}\|\eta'\|_\infty
I_{\{\varepsilon\le|\nabla f|\le 2\varepsilon\}}
\le
2\|\eta'\|_\infty
|\nabla f|^{q-1}
I_{\{\varepsilon\le|\nabla f|\le 2\varepsilon\}}.
\]

Combining the preceding estimates, we obtain
\begin{align*}
\int_{\mathbb R^n}
|\nabla f|^{r+2}&I_{\{|\nabla f|\ge 2\varepsilon\}}\,d\gamma^n
\\
&\le
\|f|\nabla f|^{r-q}\|_{L^2(\gamma^n)}
\Biggl(
2\int_{\mathbb R^n}
|\nabla f|^{2q+2}\,d\gamma^n
+
4\int_{\mathbb R^n}
|\nabla f|^{2q}\|D^2f\|_{\rm HS}^2\,d\gamma^n
\\
&+
\int_{\mathbb R^n}
\Bigl(
4\bigl(q+2\|\eta'\|_\infty I_{\{\varepsilon\le|\nabla f|\le 2\varepsilon\}}\bigr)^2
+
2(r-q)^2
\Bigr)
|\nabla f|^{2q-2}|D^2f\nabla f|^2\,d\gamma^n
\Biggr)^{1/2}.
\end{align*}
Letting $\varepsilon\to0$ and using Lebesgue's dominated convergence theorem, since
\[
4q^2+2(r-q)^2\le 4r^2
\]
for $r\ge q>0$, we obtain the announced estimate.
\end{proof}

\begin{corollary}\label{cor-1-1}
Let $n,d\in\mathbb N$, let $p>2$, and let
$f\in\mathcal P_d(\mathbb R^n)$. Then
\begin{align*}
\|\nabla f\|_{L^p(\gamma^n)}^p
&\le
2\|f|\nabla f|^{p/2-1}\|_{L^2(\gamma^n)}
\Bigl(
\int_{\mathbb R^n}
|\nabla f|^p
+
|\nabla f|^{p-2}\|D^2f\|_{\rm HS}^2
\\
&\qquad\qquad\qquad\qquad\qquad\qquad\qquad\quad
+
(p-2)^2|\nabla f|^{p-4}|D^2f\nabla f|^2\,d\gamma^n
\Bigr)^{1/2}.
\end{align*}
\end{corollary}

\begin{proof}
The estimate follows by taking
$r=p-2$ and $q=p/2-1$ in Lemma~\ref{lem-1-1}.
\end{proof}

\begin{lemma}\label{lem-1-2}
Let $k\in\mathbb{N}$.
For every $f\in C^\infty(\mathbb{R}^n)$ we have
\[
\sum_{1\le j_1,\ldots, j_k\le n }
|\nabla (\partial_{x_{j_1}}f\cdot\ldots\cdot \partial_{x_{j_k}}f)|^2
= k|\nabla f|^{2k-2}\|D^2 f\|_{\rm HS}^2
+ k(k-1)|\nabla f|^{2k-4}|D^2f \nabla f|^2.
\]
\end{lemma}	

\begin{proof}
Let
\[
u_j:=\partial_{x_j}f,\quad j=1,\ldots,n.
\]
Fix $i\in\{1,\ldots,n\}$. For a multi-index
$(j_1,\ldots,j_k)\in\{1, \ldots, n\}^k$,
we set
\[
Q_{j_1,\ldots,j_k}:=u_{j_1}\cdot\ldots\cdot u_{j_k}.
\]
Then
\[
\partial_{x_i}Q_{j_1,\ldots,j_k}
=
\sum_{\ell=1}^k
\partial_{x_i}u_{j_\ell}
\prod_{\substack{r=1\\ r\ne \ell}}^k u_{j_r}.
\]
Therefore
\[
\sum_{1\le j_1,\ldots,j_k\le n}
\bigl(\partial_{x_i}Q_{j_1,\ldots,j_k}\bigr)^2
=
\sum_{1\le j_1,\ldots,j_k\le n}
\sum_{\ell,m=1}^k
\partial_{x_i}u_{j_\ell}\partial_{x_i}u_{j_m}
\prod_{\substack{r=1\\ r\ne \ell}}^k u_{j_r}
\prod_{\substack{s=1\\ s\ne m}}^k u_{j_s}.
\]
We split the last sum into the diagonal part $\ell=m$ and the off-diagonal
part $\ell\ne m$.

For the diagonal part, for each fixed $\ell$ we get
\[
\sum_{1\le j_1,\ldots,j_k\le n}
(\partial_{x_i}u_{j_\ell})^2
\prod_{\substack{r=1\\ r\ne \ell}}^k u_{j_r}^2
=
\Bigl(\sum_{j=1}^n(\partial_{x_i}u_j)^2\Bigr)
\Bigl(\sum_{j=1}^n u_j^2\Bigr)^{k-1}.
\]
Since there are $k$ choices of $\ell$, the diagonal contribution is
\[
k|\nabla f|^{2k-2}|\nabla\partial_{x_i}f|^2.
\]
	
For the off-diagonal part, for each pair $\ell\ne m$ we get
\[
\sum_{1\le j_1,\ldots,j_k\le n}
\partial_{x_i}u_{j_\ell}\partial_{x_i}u_{j_m}
u_{j_\ell}u_{j_m}
\prod_{\substack{r=1\\ r\ne \ell,m}}^k u_{j_r}^2
=
\Bigl(\sum_{j=1}^n \partial_{x_i}u_j u_j\Bigr)^2
\Bigl(\sum_{j=1}^n u_j^2\Bigr)^{k-2}.
\]
There are $k(k-1)$ pairs $\ell\ne m$. Hence the off-diagonal
contribution is
\[
k(k-1)|\nabla f|^{2k-4}
\langle \nabla\partial_{x_i}f,\nabla f\rangle^2.
\]

Thus, for every $i=1,\ldots,n$,
\[
\sum_{1\le j_1,\ldots,j_k\le n}
\bigl(\partial_{x_i}Q_{j_1,\ldots,j_k}\bigr)^2
=
k|\nabla f|^{2k-2}|\nabla\partial_{x_i}f|^2
+
k(k-1)|\nabla f|^{2k-4}
\langle \nabla\partial_{x_i}f,\nabla f\rangle^2.
\]
Summing this identity over $i=1,\ldots,n$, we obtain
\[
\sum_{1\le j_1,\ldots,j_k\le n}
|\nabla Q_{j_1,\ldots,j_k}|^2
=
k|\nabla f|^{2k-2}
\|D^2f\|_{\rm HS}^2
+
k(k-1)|\nabla f|^{2k-4}
|D^2f\nabla f|^2.
\]
This proves the lemma.
\end{proof}
	
\begin{lemma}\label{lem-1-3}
Let $p\ge 2$ and let $n, d\in\mathbb N$. Assume that there exists a
constant $M_{p,d}$ such that
\[
\|\nabla f\|_{L^p(\gamma^n)}
\le
M_{p,d}\|f\|_{L^p(\gamma^n)}
\]
for every $f\in \mathcal P_d(\mathbb R^n)$. Then, for every $m\in\mathbb{N}$ and for every
$g=(g_1,\ldots,g_m)$ with $g_j\in\mathcal P_d(\mathbb R^n)$, one has
\[
\Bigl\|
\Bigl(\sum_{j=1}^m |\nabla g_j|^2\Bigr)^{1/2}
\Bigr\|_{L^p(\gamma^n)}
\le
\sqrt{p}\cdot M_{p,d}
\Bigl\|
\Bigl(\sum_{j=1}^m |g_j|^2\Bigr)^{1/2}
\Bigr\|_{L^p(\gamma^n)}.
\]
\end{lemma}

\begin{proof}
Let $\varepsilon_1,\ldots,\varepsilon_m$ be independent Rademacher random
variables, that is,
\[
\mathbb P(\varepsilon_i=1)
=
\mathbb P(\varepsilon_i=-1)
=
\frac12 .
\]
For each choice of signs $\varepsilon=(\varepsilon_1,\ldots,\varepsilon_m)$,
define
\[
f_\varepsilon(x):=\sum_{j=1}^m\varepsilon_j g_j(x).
\]
Then $f_\varepsilon\in\mathcal P_d(\mathbb R^n)$, and therefore, by the
assumed scalar estimate,
\[
\|\nabla f_\varepsilon\|_{L^p(\gamma^n)}^p
\le
M_{p,d}^p\|f_\varepsilon\|_{L^p(\gamma^n)}^p.
\]
Averaging over the signs and applying H\"older's inequality, we obtain
\begin{align*}
&\int_{\mathbb R^n}
\Bigl(\sum_{j=1}^m|\nabla g_j|^2
\Bigr)^{p/2}\,d\gamma^n
=
\int_{\mathbb R^n}
\Bigl(\mathbb E_\varepsilon\Bigl|
\sum_{j=1}^m\varepsilon_j\nabla g_j
\Bigr|^2\Bigr)^{p/2} \, d\gamma^n	
=
\int_{\mathbb R^n}
\mathbb E_\varepsilon
\Bigl|
\sum_{j=1}^m\varepsilon_j\nabla g_j
\Bigr|^p\,d\gamma^n
\\
&\le
\mathbb E_\varepsilon
\int_{\mathbb R^n}
\Bigl|
\sum_{j=1}^m\varepsilon_j\nabla g_j
\Bigr|^p\,d\gamma^n
\le
M_{p,d}^p
\mathbb E_\varepsilon
\int_{\mathbb R^n}
\Bigl|
\sum_{j=1}^m\varepsilon_j g_j
\Bigr|^p\,d\gamma^n
=
M_{p,d}^p
\int_{\mathbb R^n}
\mathbb E_\varepsilon\Bigl|
\sum_{j=1}^m\varepsilon_j g_j
\Bigr|^p\,d\gamma^n.
\end{align*}
By Kahane--Khintchine inequality for Rademacher sums 
(see, for instance,
\cite[Ch.~4]{LedouxTalagrand} and~\cite{HaagerupKhintchine} for the sharp constant), we have
\[
\mathbb E_\varepsilon
\Bigl|
\sum_{j=1}^m\varepsilon_j g_j
\Bigr|^p
\le
p^{p/2}
\Bigl(\sum_{j=1}^m|g_j|^2\Bigr)^{p/2}.
\]
Therefore,
\[
\int_{\mathbb R^n}
\Bigl(\sum_{j=1}^m|\nabla g_j|^2
\Bigr)^{p/2}\,d\gamma^n
\le
p^{p/2}\cdot M_{p,d}^p
\int_{\mathbb R^n}
\Bigl(\sum_{j=1}^m|g_j|^2\Bigr)^{p/2}\,d\gamma^n.
\]
Taking the $p$-th root yields
the announced estimate.
\end{proof}

\subsection{Proof of Theorem~\ref{th-1}}
Let
$k:=\lfloor p/2\rfloor$.

First suppose that $p$ is an even integer. Then $p=2k$. Since
\[
\partial_{x_{j_1}}f\cdot\ldots\cdot \partial_{x_{j_k}}f
\in \mathcal P_{k(d-1)}(\mathbb R^n)
\]
for any choice of $j_1,\ldots,j_k\in\{1,\ldots,n\}$, Lemma~\ref{lem-1-2}
and \eqref{eq-MB-2} give
\begin{align*}
&\int_{\mathbb R^n}
k|\nabla f|^{2k-2}\|D^2 f\|_{\rm HS}^2
+
k(k-1)|\nabla f|^{2k-4}|D^2f\nabla f|^2\,d\gamma^n
\\
&=
\sum_{1\le j_1,\ldots,j_k\le n}
\int_{\mathbb R^n}
\bigl|
\nabla(
\partial_{x_{j_1}}f\cdot\ldots\cdot \partial_{x_{j_k}}f)
\bigr|^2\,d\gamma^n
\\
&\le
k(d-1)
\sum_{1\le j_1,\ldots,j_k\le n}
\int_{\mathbb R^n}
|\partial_{x_{j_1}}f\cdot\ldots\cdot \partial_{x_{j_k}}f|^2\,d\gamma^n
\\
&=
k(d-1)\int_{\mathbb R^n}|\nabla f|^{2k}\,d\gamma^n.
\end{align*}
Therefore,
\begin{equation}\label{eq-HS-est}
\int_{\mathbb R^n}
|\nabla f|^{2k-2}\|D^2 f\|_{\rm HS}^2
+
(k-1)|\nabla f|^{2k-4}|D^2f\nabla f|^2\,d\gamma^n
\le
(d-1)\int_{\mathbb R^n}|\nabla f|^{2k}\,d\gamma^n.
\end{equation}
Applying Corollary~\ref{cor-1-1}, we obtain
\begin{align*}
\|\nabla f\|_{L^p(\gamma^n)}^p
&\le
2\|f|\nabla f|^{p/2-1}\|_{L^2(\gamma^n)}
\Bigl(
\int_{\mathbb R^n}
|\nabla f|^p
+
|\nabla f|^{2k-2}\|D^2f\|_{\rm HS}^2
\\
&\qquad\qquad\qquad\qquad\qquad\qquad\qquad\quad+
4(k-1)^2|\nabla f|^{2k-4}|D^2f\nabla f|^2
\,d\gamma^n
\Bigr)^{1/2}.
\end{align*}
By \eqref{eq-HS-est},
\[
\int_{\mathbb R^n}
|\nabla f|^{2k-2}\|D^2 f\|_{\rm HS}^2
+
4(k-1)^2|\nabla f|^{2k-4}|D^2f\nabla f|^2\,d\gamma^n
\le
4(k-1)(d-1)\int_{\mathbb R^n}|\nabla f|^{2k}\,d\gamma^n.
\]
Since
\[
2\sqrt{1+2(p-2)(d-1)}
\le
2\sqrt{2(p-2)d}
\le
3\sqrt{pd},
\]
it follows that
\[
\|\nabla f\|_{L^p(\gamma^n)}^p
\le
3\sqrt{pd}\,
\|f|\nabla f|^{p/2-1}\|_{L^2(\gamma^n)}
\|\nabla f\|_{L^p(\gamma^n)}^{p/2}.
\]
By H\"older's inequality,
\begin{equation}\label{eq-holder}
\|f|\nabla f|^{p/2-1}\|_{L^2(\gamma^n)}
\le
\|f\|_{L^p(\gamma^n)}
\|\nabla f\|_{L^p(\gamma^n)}^{p/2-1}.
\end{equation}
Hence
\[
\|\nabla f\|_{L^p(\gamma^n)}^p
\le
3\sqrt{pd}\,
\|f\|_{L^p(\gamma^n)}
\|\nabla f\|_{L^p(\gamma^n)}^{p-1}.
\]
If $\|\nabla f\|_{L^p(\gamma^n)}=0$, the conclusion is immediate. Otherwise,
dividing by $\|\nabla f\|_{L^p(\gamma^n)}^{p-1}$, we obtain the assertion for even integers $p$.

\medskip

It remains to consider the case where $p>4$ is not an even integer. Then
\[
2k<p<2k+2,
\quad
p/k\in(2,3).
\]
Applying Lemma~\ref{lem-1-3} to the estimate \eqref{eq-EI-est} with the
exponent $p/k$, we get
\[
\Bigl\|
\Bigl(\sum_{j=1}^m |\nabla g_j|^2\Bigr)^{1/2}
\Bigr\|_{L^{p/k}(\gamma^n)}
\le
C_1(p)d^{\frac12+\theta_p}
\Bigl\|
\Bigl(\sum_{j=1}^m |g_j|^2\Bigr)^{1/2}
\Bigr\|_{L^{p/k}(\gamma^n)}
\]
for every vector $g=(g_1,\ldots,g_m)$ whose components belong to
$\mathcal P_{k(d-1)}(\mathbb R^n)$. Here we used that
\[
\frac1\pi\arctan\Bigl(
\frac{p/k-2}{2\sqrt{p/k-1}}
\Bigr)
=
\frac1\pi\arctan\Bigl(
\frac{p-2k}{2\sqrt{k(p-k)}}
\Bigr)
=
\theta_p.
\]
We apply this estimate to the vector whose components are
\[
\partial_{x_{j_1}}f\cdot\ldots\cdot \partial_{x_{j_k}}f,
\quad
1\le j_1,\ldots,j_k\le n.
\]
Using Lemma~\ref{lem-1-2}, we obtain
\begin{align*}
&\Bigl\|
k|\nabla f|^{2k-2}\|D^2 f\|_{\rm HS}^2
+
k(k-1)|\nabla f|^{2k-4}|D^2f\nabla f|^2
\Bigr\|_{L^{p/(2k)}(\gamma^n)}^{1/2}
\\
&=
\Bigl\|
\Bigl(
\sum_{1\le j_1,\ldots,j_k\le n}
\bigl|
\nabla(
\partial_{x_{j_1}}f\cdot\ldots\cdot \partial_{x_{j_k}}f)
\bigr|^2
\Bigr)^{1/2}
\Bigr\|_{L^{p/k}(\gamma^n)}
\\
&\le
C_1(p)d^{\frac12+\theta_p}
\Bigl\|
\Bigl(
\sum_{1\le j_1,\ldots,j_k\le n}
|\partial_{x_{j_1}}f\cdot\ldots\cdot \partial_{x_{j_k}}f|^2
\Bigr)^{1/2}
\Bigr\|_{L^{p/k}(\gamma^n)}
\\
&=
C_1(p)d^{\frac12+\theta_p}
\||\nabla f|^k\|_{L^{p/k}(\gamma^n)}
=
C_1(p)d^{\frac12+\theta_p}
\|\nabla f\|_{L^p(\gamma^n)}^k.
\end{align*}
Hence
\[
\Bigl\|
k|\nabla f|^{2k-2}\|D^2 f\|_{\rm HS}^2
+
k(k-1)|\nabla f|^{2k-4}|D^2f\nabla f|^2
\Bigr\|_{L^{p/(2k)}(\gamma^n)}
\le
C_2(p) d^{1+2\theta_p}
\|\nabla f\|_{L^p(\gamma^n)}^{2k}.
\]
Since $k\ge2$,
\begin{align*}
&|\nabla f|^{p-2}\|D^2f\|_{\rm HS}^2
+
(p-2)^2|\nabla f|^{p-4}|D^2f\nabla f|^2
\\
&\le
C_3(p)|\nabla f|^{p-2k}
\Bigl(
k|\nabla f|^{2k-2}\|D^2 f\|_{\rm HS}^2
+
k(k-1)|\nabla f|^{2k-4}|D^2f\nabla f|^2
\Bigr).
\end{align*}
Therefore, by H\"older's inequality,
\begin{align*}
&\int_{\mathbb R^n}
|\nabla f|^{p-2}\|D^2f\|_{\rm HS}^2
+
(p-2)^2|\nabla f|^{p-4}|D^2f\nabla f|^2
\,d\gamma^n
\\
&\le
C_3(p)
\bigl\||\nabla f|^{p-2k}\bigr\|_{L^{p/(p-2k)}(\gamma^n)}
\Bigl\|
k|\nabla f|^{2k-2}\|D^2 f\|_{\rm HS}^2
+
k(k-1)|\nabla f|^{2k-4}|D^2f\nabla f|^2
\Bigr\|_{L^{p/(2k)}(\gamma^n)}
\\
&\le
C_4(p)d^{1+2\theta_p}
\|\nabla f\|_{L^p(\gamma^n)}^{p-2k}
\|\nabla f\|_{L^p(\gamma^n)}^{2k}
=
C_4(p)d^{1+2\theta_p}
\|\nabla f\|_{L^p(\gamma^n)}^p.
\end{align*}
Substituting this estimate into Corollary~\ref{cor-1-1}, gives
\[
\|\nabla f\|_{L^p(\gamma^n)}^p
\le
C_5(p)d^{\frac12+\theta_p}
\|f|\nabla f|^{p/2-1}\|_{L^2(\gamma^n)}
\|\nabla f\|_{L^p(\gamma^n)}^{p/2}.
\]
By \eqref{eq-holder},
\[
\|\nabla f\|_{L^p(\gamma^n)}^p
\le
C_5(p)d^{\frac12+\theta_p}
\|f\|_{L^p(\gamma^n)}
\|\nabla f\|_{L^p(\gamma^n)}^{p-1}.
\]
If $\|\nabla f\|_{L^p(\gamma^n)}=0$, the conclusion is immediate. Otherwise,
dividing by $\|\nabla f\|_{L^p(\gamma^n)}^{p-1}$ proves the estimate for
non-even $p>4$. 

\medskip
Finally, we estimate the error term $\theta_p$
for $p\ge 4$.
We write
\[
p=2k+\alpha,\quad 0\le \alpha<2, \quad k:=\lfloor p/2\rfloor.
\]
For $p\ge4$ one has $k\ge2$, and
\[
\theta_p
=
\frac1\pi
\arctan\Bigl(
\frac{\alpha}{2\sqrt{k(k+\alpha)}}
\Bigr).
\]
Our goal is to prove that $p\theta_p\le 2/3$. For fixed $k$, the function
\[
\alpha\mapsto
(2k+\alpha)
\arctan\Bigl(
\frac{\alpha}{2\sqrt{k(k+\alpha)}}
\Bigr)
=
(2k+\alpha)
\arctan\Bigl(
\frac{1}{2\sqrt{k/\alpha(k/\alpha+1)}}
\Bigr)
\]
is increasing on $[0,2]$. 
Therefore,
\[
p\theta_p
\le
\frac{2(k+1)}{\pi}
\arctan\Bigl(
\frac{1}{\sqrt{k(k+2)}}
\Bigr)
=\frac{2(k+1)}{\pi}
\arcsin\frac{1}{k+1}.
\]
The function
\[
x\mapsto x\arcsin\frac1x
\]
is decreasing on $(1,\infty)$, since
\[
\frac{d}{dx}\biggl(x\arcsin\frac1x\biggr)
=
\arcsin\frac1x-\frac{1}{\sqrt{x^2-1}}<0.
\]
Thus, as $k+1\ge3>1$,
\[
(k+1)\arcsin\frac{1}{k+1}
\le
3\arcsin\frac13.
\]
Finally,
\[
3\arcsin\frac13<\frac{\pi}{3}.
\]
Indeed, if $a=\arcsin(1/3)$, then
\[
\sin(3a)=3\sin a-4\sin^3a
=
1-\frac{4}{27}
=
\frac{23}{27}
<
\frac{\sqrt3}{2}
=
\sin\frac{\pi}{3},
\]
and $0<3a<\pi/2$. Therefore
\[
p\theta_p
\le
\frac{2}{\pi}\cdot \frac{\pi}{3}
=
\frac23,
\]
which gives the claimed estimate.
\qed

\section{The uniform distribution on the cube}

\subsection{Integration by parts on the unit cube}

For two vector fields $u,v\in C^\infty(\mathbb R^n,\mathbb R^n)$, let
\[
\langle u,v\rangle_0
:=
\sum_{j=1}^n(1-x_j^2)u_jv_j
\]
and let
\[
|u|_0^2:=\langle u,u\rangle_0.
\]
For $r>0$, we set, on $[-1,1]^n$,
\[
|u|_0^r:=\bigl(|u|_0^2\bigr)^{r/2}.
\]

For $f\in C^\infty(\mathbb R^n)$, set
\[
\delta_{0,x_j}f
:=
2x_jf-(1-x_j^2)\partial_{x_j}f.
\]
For a vector field $u=(u_1,\ldots,u_n)\in C^\infty(\mathbb R^n,\mathbb R^n)$, set
\[
\delta_0u:=\sum_{j=1}^n\delta_{0,x_j}u_j.
\]

For $f\in C^\infty(\mathbb R^n)$ and
$u\in C^\infty(\mathbb R^n,\mathbb R^n)$, we have
\begin{equation}\label{jacobi-int-by-parts-1}
\int_{[-1,1]^n}
(1-x_j^2)u_j\partial_{x_j}f\,dx
=
\int_{[-1,1]^n}
f\delta_{0,x_j}u_j\,dx.
\end{equation}
Therefore,
\begin{equation}\label{jacobi-int-by-parts}
\int_{[-1,1]^n}\langle \nabla f,u\rangle_0\,dx
=
\int_{[-1,1]^n}f\,\delta_0u\,dx.
\end{equation}

Let
\[
L_0 f
:=
\sum_{j=1}^n
\Bigl(
(1-x_j^2)\partial_{x_jx_j}^2f
-
2x_j\partial_{x_j}f
\Bigr).
\]
In particular,
\[
\delta_0\nabla f=-L_0f.
\]

\subsection{Preliminary lemmas}

\begin{lemma}\label{lem-2-1}
Let $D\in \mathbb N$ and let $Q\in \mathcal P_D(\mathbb R)$ be such that $Q(t)\ge 0$ $\forall t\in[-1, 1]$. Then
\[
\int_{-1}^1 Q(t)\,dt
\le
16D^2
\int_{-1}^1(1-t^2)Q(t)\,dt .
\]
\end{lemma}

\begin{proof}
By the standard Nikolskii inequality
(see, for example, \cite[Chapter~4, Theorem~2.6]{DL-book}), we have
\[
\max_{t\in[-1,1]} Q(t)
\le
2D^2 \int_{-1}^1 Q(t)\,dt .
\]
Then, for any $\delta\in(0,1)$,
\[
\int_{\{1-\delta\le |t|\le 1\}}Q(t)\,dt
\le
2\delta \max_{t\in[-1,1]}Q(t)
\le
4\delta D^2\int_{-1}^1Q(t)\,dt.
\]
Let
\[
\delta:=\frac{1}{8D^2}.
\]
Then
\[
\int_{\{|t|\le 1-\delta\}}Q(t)\,dt
\ge
\frac12\int_{-1}^1Q(t)\,dt.
\]
Therefore
\[
\int_{-1}^1(1-t^2)Q(t)\,dt
\ge
\delta\int_{\{|t|\le 1-\delta\}}Q(t)\,dt
\ge
\frac{\delta}{2}\int_{-1}^1Q(t)\,dt.
\]
Since $\delta=1/(8D^2)$, this gives
\[
\int_{-1}^1Q(t)\,dt
\le
16D^2
\int_{-1}^1(1-t^2)Q(t)\,dt.
\]
This is the announced estimate.
\end{proof}

\begin{lemma}\label{lem-2-2}
For every $n,k,d\in\mathbb N$ and every
$f\in\mathcal P_d(\mathbb R^n)$, one has
\[
\|\nabla f\|_{L^{2k}([-1,1]^n)}
\le
8kd\,
\||\nabla f|_0\|_{L^{2k}([-1,1]^n)}.
\]
\end{lemma}

\begin{proof}
We first record a one-dimensional consequence of Lemma~\ref{lem-2-1}.
Let $D,r,m\in\mathbb N\cup\{0\}$ and let $Q\in\mathcal P_D(\mathbb R)$.
Then
\begin{equation}\label{eq-cor}
\int_{-1}^1(1-t^2)^r|Q(t)|^2\,dt
\le
64^m(D+r+m)^{2m}
\int_{-1}^1(1-t^2)^{r+m}|Q(t)|^2\,dt.
\end{equation}
Indeed, for $m=0$, this estimate is just the identity.
For $m\ge 1$,
applying Lemma~\ref{lem-2-1} to the non-negative polynomial
\[
(1-t^2)^r|Q(t)|^2,
\]
whose degree is at most $2(D+r)$, gives
\[
\int_{-1}^1(1-t^2)^r|Q(t)|^2\,dt
\le
64(D+r+1)^2
\int_{-1}^1(1-t^2)^{r+1}|Q(t)|^2\,dt.
\]
The shift $+1$ covers the case $D=r=0$. Iterating this estimate $m$ times gives
\eqref{eq-cor}.

For $(j_1,\ldots,j_k)\in\{1,\ldots,n\}^k$, set
\[
Q_{j_1,\ldots,j_k}:=\partial_{x_{j_1}}f\cdot\ldots\cdot\partial_{x_{j_k}}f\in
\mathcal P_{k(d-1)}(\mathbb R^n).
\]
For $j\in\{1,\ldots,n\}$, let $m_j=m_j((j_1,\ldots,j_k))$ be the number of occurrences of $j$ in $(j_1,\ldots,j_k)$. Then
\[
\sum_{j=1}^n m_j=k
\]
and
\[
\prod_{\ell=1}^k(1-x_{j_\ell}^2)
=
\prod_{j=1}^n(1-x_j^2)^{m_j}.
\]
Applying \eqref{eq-cor} successively in every variable gives
\begin{align*}
\int_{[-1,1]^n}|Q_{j_1,\ldots,j_k}|^2\,dx
&\le
\prod_{j=1}^n
64^{m_j}\bigl(k(d-1)+m_j\bigr)^{2m_j}
\int_{[-1,1]^n}
\prod_{j=1}^n(1-x_j^2)^{m_j}|Q_{j_1,\ldots,j_k}|^2\,dx
\\
&\le
64^k(kd)^{2k}
\int_{[-1,1]^n}
\prod_{\ell=1}^k(1-x_{j_\ell}^2)|Q_{j_1,\ldots,j_k}|^2\,dx
\\
&=
64^k(kd)^{2k}
\int_{[-1,1]^n}
(1-x_{j_1}^2)(\partial_{x_{j_1}}f)^2
\cdot\ldots\cdot
(1-x_{j_k}^2)(\partial_{x_{j_k}}f)^2\,dx.
\end{align*}
Now summing over all $(j_1,\ldots,j_k)$, and using
\[
|\nabla f|^{2k}
=
\sum_{1\le j_1,\ldots,j_k\le n}
\bigl(\partial_{x_{j_1}}f\cdot\ldots\cdot\partial_{x_{j_k}}f\bigr)^2
\]
and
\[
|\nabla f|_0^{2k}
=
\sum_{1\le j_1,\ldots,j_k\le n}
(1-x_{j_1}^2)(\partial_{x_{j_1}}f)^2
\cdot\ldots\cdot
(1-x_{j_k}^2)(\partial_{x_{j_k}}f)^2,
\]
we obtain
\[
\int_{[-1,1]^n}|\nabla f|^{2k}\,dx
\le
(8kd)^{2k}
\int_{[-1,1]^n}|\nabla f|_0^{2k}\,dx.
\]
Taking the power $1/(2k)$ gives the claimed estimate.
\end{proof}

\begin{lemma}\label{lem-2-3}
Let $u\in C^\infty(\mathbb R^n, \mathbb{R}^n)$. Then
\[
\int_{[-1, 1]^n}(\delta_0 u)^2\,dx
\le
2
\int_{[-1, 1]^n}
|u|_0^2\,dx
+
\int_{[-1, 1]^n}
\sum_{j,k=1}^n
(1-x_j^2)(1-x_k^2)
(\partial_{x_j}u_k)^2\,dx.
\]
\end{lemma}

\begin{proof}
Applying the integration by parts formula \eqref{jacobi-int-by-parts}
with $f=\delta_0 u\in C^\infty(\mathbb{R}^n)$ and the vector field $u$, we get
\[
\int_{[-1,1]^n}(\delta_0 u)^2\,dx
=
\sum_{j,k=1}^n
\int_{[-1,1]^n}
(1-x_j^2)u_j
\partial_{x_j}(\delta_{0, x_k}u_k)
\,dx.
\]

When $j\ne k$, we have
\[
\partial_{x_j}(\delta_{0, x_k}u_k)=
\partial_{x_j}
\bigl(
2x_ku_k
-
(1-x_k^2)\partial_{x_k}u_k
\bigr)
=
2x_k\partial_{x_j}u_k
-
(1-x_k^2)\partial_{x_k}\partial_{x_j}u_k
=\delta_{0, x_k}(\partial_{x_j}u_k).
\]
Applying the integration by parts formula
\eqref{jacobi-int-by-parts-1}
in the variable
$x_k$, we obtain
\begin{align*}
\int_{[-1,1]^n}
(1-x_j^2)u_j
\partial_{x_j}(\delta_{0, x_k}u_k)
\,dx
&=
\int_{[-1,1]^n}
(1-x_j^2)u_j\delta_{0, x_k}(\partial_{x_j}u_k)
\,dx
\\
&=
\int_{[-1,1]^n}
(1-x_j^2)(1-x_k^2)
\partial_{x_k}u_j\,\partial_{x_j}u_k\,dx.
\end{align*}

Now let $j=k$. Then
\begin{align*}
&
\partial_{x_j}(\delta_{0, x_j}u_j)
=\partial_{x_j}
\bigl(
2x_ju_j
-
(1-x_j^2)\partial_{x_j}u_j
\bigr)
\\
&=
2u_j
+
4x_j\partial_{x_j}u_j
-
(1-x_j^2)\partial_{x_jx_j}^2u_j 
= 2u_j
+
2x_j\partial_{x_j}u_j + \delta_{0, x_j}(\partial_{x_j}u_j).
\end{align*}
Therefore,
\begin{align*}
&\int_{[-1,1]^n}
(1-x_j^2)u_j
\partial_{x_j}(\delta_{0, x_j}u_j)\,dx
=
\int_{[-1,1]^n}
(1-x_j^2)u_j\bigl(
2u_j
+
2x_j\partial_{x_j}u_j + \delta_{0, x_j}(\partial_{x_j}u_j)\bigr)
\,dx
\\
&=
2
\int_{[-1,1]^n}
(1-x_j^2)u_j^2\,dx
+
\int_{[-1,1]^n}
(1-x_j^2)u_j
\bigl(
2x_j\partial_{x_j}u_j
+ \delta_{0, x_j}(\partial_{x_j}u_j)\bigr)\,dx.
\end{align*}
Applying the integration by parts formula \eqref{jacobi-int-by-parts-1} in the second integral gives
\begin{align*}
&\int_{[-1,1]^n}
(1-x_j^2)u_j
\bigl(
2x_j\partial_{x_j}u_j
+ \delta_{0, x_j}(\partial_{x_j}u_j)\bigr)\,dx
\\
&=
\int_{[-1,1]^n}
(1-x_j^2)u_j
2x_j\partial_{x_j}u_j
+(1-x_j^2)\partial_{x_j}\bigl((1-x_j^2)u_j\bigr)
\partial_{x_j}u_j
\,dx
\\
&=
\int_{[-1,1]^n}
(1-x_j^2)^2(\partial_{x_j}u_j)^2\,dx.
\end{align*}
Combining the diagonal and off-diagonal terms, we obtain
\begin{align*}
\int_{[-1,1]^n}(\delta_0 u)^2\,dx
&=
2
\int_{[-1,1]^n}
|u|_0^2\,dx
+
\int_{[-1,1]^n}
\sum_{j,k=1}^n
(1-x_j^2)(1-x_k^2)
\partial_{x_j}u_k\,\partial_{x_k}u_j\,dx.
\end{align*}
By the Cauchy--Schwarz inequality,
\begin{align*}
&\sum_{j,k=1}^n
(1-x_j^2)(1-x_k^2)
\partial_{x_j}u_k\,\partial_{x_k}u_j
\\
&=
\sum_{j,k=1}^n
\bigl((1-x_j^2)^{1/2}(1-x_k^2)^{1/2}
\partial_{x_j}u_k\bigr)
\bigl((1-x_j^2)^{1/2}(1-x_k^2)^{1/2}
\partial_{x_k}u_j\bigr)
\\
&\le
\sum_{j,k=1}^n
(1-x_j^2)(1-x_k^2)(\partial_{x_j}u_k)^2.
\end{align*}
Thus,
\begin{align*}
\int_{[-1,1]^n}(\delta_0 u)^2\,dx
&\le
2
\int_{[-1,1]^n}
|u|_0^2\,dx+
\int_{[-1,1]^n}
\sum_{j,k=1}^n
(1-x_j^2)(1-x_k^2)
(\partial_{x_j}u_k)^2\,dx.
\end{align*}
This is the announced estimate.
\end{proof}

\begin{lemma}\label{lem-2-4}
Let $\alpha_1,\ldots,\alpha_n>-1$ and let
\[
\mu_{\alpha_j}(dt):=c_{\alpha_j}(1-t^2)^{\alpha_j}\,dt
\]
be the corresponding probability measures on $[-1,1]$. Set
\[
\mu_\alpha(dx)
:=
\bigotimes_{j=1}^n \mu_{\alpha_j}(dx_j)
=
\prod_{j=1}^n c_{\alpha_j}(1-x_j^2)^{\alpha_j}\,dx
\]
and
\[
\alpha_*:=\max_{1\le j\le n}\alpha_j.
\]
Then, for every $f\in\mathcal P_d(\mathbb R^n)$,
\[
\int_{[-1,1]^n}
|\nabla f|_0^2
\,d\mu_\alpha
\le
d(d+2\alpha_*+1)
\int_{[-1,1]^n}f^2\,d\mu_\alpha .
\]
\end{lemma}

\begin{proof}
For each $j\in\{1,\ldots,n\}$, consider the one-dimensional Jacobi operator
\[
L_{\alpha_j}h(t)
=
(1-t^2)h''(t)-2(\alpha_j+1)t h'(t).
\]
An integration by parts gives, for polynomials $h_1,h_2$,
\[
\int_{-1}^1(1-t^2)h_1'(t)h_2'(t)\,\mu_{\alpha_j}(dt)
=
-\int_{-1}^1 h_1 L_{\alpha_j}h_2\,d\mu_{\alpha_j}.
\]
Define
\[
L_\alpha f
:=
\sum_{j=1}^n
\bigl(
(1-x_j^2)\partial_{x_jx_j}^2f
-
2(\alpha_j+1)x_j\partial_{x_j}f
\bigr).
\]
Applying the preceding integration by parts formula in each coordinate gives
\[
\int_{[-1,1]^n}
|\nabla f|_0^2\,d\mu_\alpha
=
\int_{[-1,1]^n}
\sum_{j=1}^n(1-x_j^2)(\partial_{x_j}f)^2\,d\mu_\alpha
=
-\int_{[-1,1]^n} f L_\alpha f\,d\mu_\alpha.
\]

Let $\{u_m^{(\alpha_j)}\colon m\in\mathbb N\cup\{0\}\}$ be the system
obtained by normalizing the Jacobi polynomials
$P_m^{(\alpha_j,\alpha_j)}$ in $L^2(\mu_{\alpha_j})$. Then
$\{u_m^{(\alpha_j)}\colon m\in\mathbb N\cup\{0\}\}$ is an orthonormal
system in $L^2(\mu_{\alpha_j})$. The Jacobi differential equation (see~\cite[Chapter IV, formula (4.2.1)]{SzegoOP})
gives
\[
L_{\alpha_j}u_m^{(\alpha_j)}
=
-m(m+2\alpha_j+1)u_m^{(\alpha_j)}.
\]

For a multi-index ${\bf m}=(m_1,\ldots,m_n)$, $m_j\in\mathbb N\cup\{0\}$, set
\[
U_{\bf m}(x):=
\prod_{j=1}^n u_{m_j}^{(\alpha_j)}(x_j),
\quad
|{\bf m}|:=m_1+\ldots+m_n.
\]
The system $\{U_{\bf m}\colon |{\bf m}|\le d\}$ is an orthonormal basis of
$\mathcal P_d(\mathbb R^n)$ with respect to the inner product of
$L^2(\mu_\alpha)$.
Moreover,
\[
L_\alpha U_{\bf m}
=
-\lambda_{\bf m}U_{\bf m},
\]
where
\[
\lambda_{\bf m}
=
\sum_{j=1}^n m_j(m_j+2\alpha_j+1).
\]

Since $f\in\mathcal P_d(\mathbb R^n)$, we may write
\[
f=\sum_{|{\bf m}|\le d}a_{\bf m}U_{\bf m}.
\]
Therefore
\[
\int_{[-1,1]^n}f^2\,d\mu_\alpha
=
\sum_{|{\bf m}|\le d}|a_{\bf m}|^2
\]
and
\[
-\int_{[-1,1]^n}fL_\alpha f\,d\mu_\alpha
=
\sum_{|{\bf m}|\le d}\lambda_{\bf m}|a_{\bf m}|^2.
\]

Now we note that, for $|{\bf m}|\le d$, one has
\[
\lambda_{\bf m}
=
\sum_{j=1}^n m_j(m_j+2\alpha_j+1)
\le
(d+2\alpha_*+1)\sum_{j=1}^n m_j
=
(d+2\alpha_*+1)|{\bf m}|
\le
d(d+2\alpha_*+1).
\]
This implies the announced estimate.
\end{proof}

\subsection{Key lemmas}

\begin{lemma}\label{lem-2-5}
Let $r\ge q\ge 1$ and let $f\in\mathcal P_d(\mathbb R^n)$. Then
\begin{align*}
&\int_{[-1,1]^n}
|\nabla f|_0^{r+2}\,dx
\le
2\bigl\|f|\nabla f|_0^{r-q}\bigr\|_{L^2([-1, 1]^n)}
\biggl(
\int_{[-1,1]^n}|\nabla f|_0^{2q+2}\,dx
\\
&+
\int_{[-1,1]^n}|\nabla f|_0^{2q}
\sum_{j,k=1}^n
(1-x_j^2)(1-x_k^2)
(\partial_{x_jx_k}^2f)^2\,dx
+
\frac{r^2}{4}\int_{[-1,1]^n}
|\nabla f|_0^{2q-2}
\bigl|\nabla |\nabla f|_0^2\bigr|_0^2
\,dx
\biggr)^{1/2}.
\end{align*}
\end{lemma}

\begin{proof}
The proof is parallel to the proof of the Gaussian counterpart,
Lemma~\ref{lem-1-1}.

Let $\eta\in C^\infty(\mathbb R)$ be such that
\[
\eta(t)=0 \quad \text{for } t\le1,
\quad
\eta(t)=1 \quad \text{for } t\ge2,
\quad
0\le \eta(t)\le 1 \quad \text{for all } t\in\mathbb R.
\]
For $\varepsilon>0$, set
$\eta_\varepsilon(t):=\eta(t/\varepsilon)$.

For $a\ge 1$, we consider the vector field
\[
u_{a,\varepsilon}
:=
|\nabla f|_0^a\eta_\varepsilon(|\nabla f|_0^2)\nabla f.
\]
We understand the coefficient
\[
|\nabla f|_0^a\eta_\varepsilon(|\nabla f|_0^2)
\]
as the composition of the polynomial $|\nabla f|_0^2$ with the smooth function
on $\mathbb R$ which equals $s^{a/2}\eta_\varepsilon(s)$ for $s\ge0$ and is
zero for $s<0$. Thus, 
\[
u_{a,\varepsilon}
:=
|\nabla f|_0^a\eta_\varepsilon(|\nabla f|_0^2)\nabla f\in C^\infty(\mathbb R^n,\mathbb R^n).
\]

We have
\begin{align*}
&\delta_0 u_{r,\varepsilon}
=
\sum_{j=1}^n
\bigl(2x_ju_{r,\varepsilon,j}
-
(1-x_j^2)\partial_{x_j}u_{r,\varepsilon,j}\bigr)
\\
&=
-|\nabla f|_0^r\eta_\varepsilon(|\nabla f|_0^2)L_0 f
-
\Bigl(
\frac r2 |\nabla f|_0^{r-2}\eta_\varepsilon(|\nabla f|_0^2)
+
|\nabla f|_0^r\eta_\varepsilon'(|\nabla f|_0^2)
\Bigr)\langle \nabla |\nabla f|_0^2, \nabla f\rangle_0.
\end{align*}
Therefore,
\begin{align*}
\delta_0 u_{r,\varepsilon}
&=
|\nabla f|_0^{r-q}
\Bigl(
\delta_0 u_{q,\varepsilon}
-
\frac{r-q}{2}
|\nabla f|_0^{q-2}\eta_\varepsilon(|\nabla f|_0^2)
\langle \nabla|\nabla f|_0^2, \nabla f\rangle_0
\Bigr).
\end{align*}

By the integration by parts formula \eqref{jacobi-int-by-parts},
\[
\int_{[-1,1]^n}
|\nabla f|_0^{r+2}I_{\{|\nabla f|_0^2\ge2\varepsilon\}}\,dx
\le
\int_{[-1,1]^n}\langle \nabla f, u_{r,\varepsilon}\rangle_0\,dx
=
\int_{[-1,1]^n}f\,\delta_0 u_{r,\varepsilon}\,dx.
\]
Hence, by the Cauchy--Schwarz inequality,
\begin{align}\label{eq-CS-est}
&\int_{[-1,1]^n}
|\nabla f|_0^{r+2}I_{\{|\nabla f|_0^2\ge2\varepsilon\}}\,dx
\\
&\le
\bigl\|f|\nabla f|_0^{r-q}\bigr\|_{L^2}
\biggl(
\int_{[-1,1]^n}
\Bigl(
\delta_0 u_{q,\varepsilon}
-
\frac{r-q}{2}
|\nabla f|_0^{q-2}\eta_\varepsilon(|\nabla f|_0^2)
\langle \nabla|\nabla f|_0^2, \nabla f\rangle_0
\Bigr)^2
\,dx
\biggr)^{1/2}.\notag
\end{align}
The last integral is bounded by
\begin{equation}\label{eq-lem-2-4-middle}
2\int_{[-1,1]^n}(\delta_0 u_{q,\varepsilon})^2\,dx
+
\frac{(r-q)^2}{2}
\int_{[-1,1]^n}
|\nabla f|_0^{2q-2}
\bigl|\nabla|\nabla f|_0^2\bigr|_0^2
\,dx.
\end{equation}
Applying Lemma~\ref{lem-2-3} to $u_{q,\varepsilon}$ in the first term gives
\[
\int_{[-1,1]^n}(\delta_0 u_{q,\varepsilon})^2\,dx
\le
2
\int_{[-1,1]^n}|u_{q,\varepsilon}|_0^2\,dx
+
\int_{[-1,1]^n}
\sum_{j,k=1}^n
(1-x_j^2)(1-x_k^2)
(\partial_{x_j}u_{q,\varepsilon,k})^2\,dx.
\]
By the definition of $u_{q, \varepsilon}$,
\[
\int_{[-1,1]^n}|u_{q,\varepsilon}|_0^2\,dx
\le
\int_{[-1,1]^n}|\nabla f|_0^{2q+2}\,dx.
\]
Next,
\[
u_{q,\varepsilon,k}
=
|\nabla f|_0^q\eta_\varepsilon(|\nabla f|_0^2)\partial_{x_k}f,
\]
and
\[
\partial_{x_j}u_{q,\varepsilon,k}
=
|\nabla f|_0^q\eta_\varepsilon(|\nabla f|_0^2)\partial_{x_jx_k}^2f
+
\Bigl(
\frac q2 |\nabla f|_0^{q-2}\eta_\varepsilon(|\nabla f|_0^2)
+
|\nabla f|_0^q\eta_\varepsilon'(|\nabla f|_0^2)
\Bigr)
\partial_{x_j}\bigl(|\nabla f|^2_0\bigr)\,\partial_{x_k}f.
\]
On the cube $[-1, 1]^n$ we have
\[
\Bigl|
\frac q2 |\nabla f|_0^{q-2}\eta_\varepsilon(|\nabla f|_0^2)
+
|\nabla f|_0^q\eta_\varepsilon'(|\nabla f|_0^2)
\Bigr|
\le \Bigl(\frac q2+2\|\eta'\|_\infty
I_{\{\varepsilon\le |\nabla f|_0^2\le 2\varepsilon\}}\Bigr) |\nabla f|_0^{q-2}.
\]
Therefore, on the cube $[-1, 1]^n$,
\begin{align*}
&\sum_{j,k=1}^n
(1-x_j^2)(1-x_k^2)
(\partial_{x_j}u_{q,\varepsilon,k})^2
\\
&\le
2|\nabla f|_0^{2q}
\sum_{j,k=1}^n
(1-x_j^2)(1-x_k^2)
(\partial_{x_jx_k}^2f)^2
\\
&+
2\Bigl(\frac q2+2\|\eta'\|_\infty
I_{\{\varepsilon\le |\nabla f|_0^2\le 2\varepsilon\}}\Bigr)^2|\nabla f|_0^{2q-4}
\sum_{j,k=1}^n
(1-x_j^2)(1-x_k^2)
\bigl(\partial_{x_j}(|\nabla f|_0^2)\bigr)^2(\partial_{x_k}f)^2
\\
&=
2|\nabla f|_0^{2q}
\sum_{j,k=1}^n
(1-x_j^2)(1-x_k^2)
(\partial_{x_jx_k}^2f)^2
+
2\Bigl(\frac q2+2\|\eta'\|_\infty
I_{\{\varepsilon\le |\nabla f|_0^2\le 2\varepsilon\}}\Bigr)^2|\nabla f|_0^{2q-2}
\bigl|\nabla |\nabla f|_0^2\bigr|_0^2.
\end{align*}
Thus,
\begin{align*}
\int_{[-1,1]^n}(\delta_0 u_{q,\varepsilon})^2\,dx
&\le
2\int_{[-1,1]^n}|\nabla f|_0^{2q+2}\,dx
+
2\int_{[-1,1]^n}|\nabla f|_0^{2q}
\sum_{j,k=1}^n
(1-x_j^2)(1-x_k^2)
(\partial_{x_jx_k}^2f)^2\,dx
\\
&+
2\int_{[-1,1]^n}
\Bigl(\frac q2+2\|\eta'\|_\infty
I_{\{\varepsilon\le |\nabla f|_0^2\le 2\varepsilon\}}\Bigr)^2|\nabla f|_0^{2q-2}
\bigl|\nabla |\nabla f|_0^2\bigr|_0^2
\,dx.
\end{align*}

Substituting this estimate into
\eqref{eq-lem-2-4-middle}, and using $q\le r$, we obtain
\begin{align*}
&\int_{[-1,1]^n}
\Bigl(
\delta_0 u_{q,\varepsilon}
-
\frac{r-q}{2}
|\nabla f|_0^{q-2}\eta_\varepsilon(|\nabla f|_0^2)
\langle \nabla|\nabla f|_0^2, \nabla f\rangle_0
\Bigr)^2
\,dx
\\
&\le
4\biggl(
\int_{[-1,1]^n}|\nabla f|_0^{2q+2}\,dx
+
\int_{[-1,1]^n}|\nabla f|_0^{2q}
\sum_{j,k=1}^n
(1-x_j^2)(1-x_k^2)
(\partial_{x_jx_k}^2f)^2\,dx
\\
&+\int_{[-1,1]^n}
\Bigl(\frac r2+2\|\eta'\|_\infty
I_{\{\varepsilon\le |\nabla f|_0^2\le 2\varepsilon\}}\Bigr)^2|\nabla f|_0^{2q-2}
\bigl|\nabla |\nabla f|_0^2\bigr|_0^2
\,dx
\biggr).
\end{align*}
Returning to \eqref{eq-CS-est} and letting $\varepsilon\to0$, the dominated
convergence theorem gives the claimed estimate.
\end{proof}

\begin{lemma}\label{lem-2-6}
Let $k\in\mathbb N$, $k\ge 2$, and let $f\in\mathcal P_d(\mathbb R^n)$. Then
\begin{align*}
&\int_{[-1,1]^n}|\nabla f|_0^{2k-2}
\sum_{i,j=1}^n
(1-x_i^2)(1-x_j^2)
(\partial_{x_ix_j}^2f)^2\,dx
+
(k-1)^2
\int_{[-1,1]^n}
|\nabla f|_0^{2k-4}
\bigl|\nabla |\nabla f|_0^2\bigr|_0^2\,dx
\\
&\le
512k^4d^2
\bigl\||\nabla f|_0\bigr\|_{L^{2k}([-1, 1]^n)}^{2k}.
\end{align*}
\end{lemma}

\begin{proof}
For
$(j_1,\ldots,j_k)\in\{1,\ldots,n\}^k$, set
\[
Q_{j_1,\ldots,j_k}:=
\partial_{x_{j_1}}f\cdot\ldots\cdot\partial_{x_{j_k}}f
\in\mathcal P_{k(d-1)}(\mathbb R^n).
\]
For $j\in\{1,\ldots,n\}$, let $m_j=m_j((j_1,\ldots,j_k))$ be the number of occurrences of
$j$ in $(j_1,\ldots,j_k)$. Then
\[
\sum_{j=1}^n m_j=k
\quad\text{and}\quad
\prod_{\ell=1}^k(1-x_{j_\ell}^2) = \prod_{j=1}^n(1-x_j^2)^{m_j}.
\]
By Lemma \ref{lem-2-4},
applied to the polynomial $Q_{j_1,\ldots,j_k}$
and the measure $\mu_{\bf m}$, ${\bf m}=(m_1, \ldots, m_n)$, we obtain
\begin{align*}
&\int_{[-1,1]^n}
\prod_{\ell=1}^k(1-x_{j_\ell}^2)
\sum_{i=1}^n(1-x_i^2)(\partial_{x_i}Q_{j_1,\ldots,j_k})^2\,dx
\\
&\le
k(d-1)\bigl(k(d-1) + 2\max_{1\le j\le n}m_j+1\bigr)
\int_{[-1,1]^n}
\prod_{\ell=1}^k(1-x_{j_\ell}^2)Q_{j_1,\ldots,j_k}^2\,dx
\\
&\le
2k^2d^2
\int_{[-1,1]^n}
\prod_{\ell=1}^k(1-x_{j_\ell}^2)Q_{j_1,\ldots,j_k}^2\,dx.
\end{align*}
Summing over all $(j_1,\ldots,j_k)$, we obtain
\begin{equation}\label{eq-MB-weighted}
\sum_{j_1,\ldots,j_k=1}^n
\int_{[-1,1]^n}
\prod_{\ell=1}^k(1-x_{j_\ell}^2)
\sum_{i=1}^n(1-x_i^2)
\bigl(\partial_{x_i}Q_{j_1,\ldots,j_k}\bigr)^2\,dx
\le
2k^2d^2
\int_{[-1,1]^n}|\nabla f|_0^{2k}\,dx.
\end{equation}

Now we compute the left hand side.
For fixed $i$, we have
\[
\partial_{x_i}Q_{j_1,\ldots,j_k}
=
\sum_{s=1}^k
\partial_{x_i x_{j_s}}^2f
\prod_{\substack{\ell=1\\ \ell\ne s}}^k
\partial_{x_{j_\ell}}f.
\]
Hence
\[
\bigl(\partial_{x_i}Q_{j_1,\ldots,j_k}\bigr)^2
=
\sum_{s=1}^k
\bigl(\partial_{x_i x_{j_s}}^2f\bigr)^2
\prod_{\substack{\ell=1\\ \ell\ne s}}^k
\bigl(\partial_{x_{j_\ell}}f\bigr)^2
+
\sum_{\substack{s,t=1\\ s\ne t}}^k
\partial_{x_i x_{j_s}}^2f\,
\partial_{x_i x_{j_t}}^2f
\partial_{x_{j_s}}f\,\partial_{x_{j_t}}f
\prod_{\substack{\ell=1\\ \ell\ne s,t}}^k
\bigl(\partial_{x_{j_\ell}}f\bigr)^2.
\]
First, the diagonal part gives
\begin{align*}
&\sum_{j_1,\ldots,j_k=1}^n
\prod_{\ell=1}^k(1-x_{j_\ell}^2)
\sum_{s=1}^k
\bigl(\partial_{x_i x_{j_s}}^2f\bigr)^2
\prod_{\substack{\ell=1\\ \ell\ne s}}^k
\bigl(\partial_{x_{j_\ell}}f\bigr)^2
\\
&=
k
\biggl(\sum_{j=1}^n
(1-x_j^2)
\bigl(\partial_{x_i x_j}^2f\bigr)^2\biggr)
\biggl(
\sum_{m=1}^n
(1-x_m^2)(\partial_{x_m}f)^2
\biggr)^{k-1}
\\
&=
k|\nabla f|_0^{2k-2}
\sum_{j=1}^n
(1-x_j^2)
\bigl(\partial_{x_i x_j}^2f\bigr)^2 .
\end{align*}
Second, the off-diagonal part gives
\begin{align*}
&\sum_{j_1,\ldots,j_k=1}^n
\prod_{\ell=1}^k(1-x_{j_\ell}^2)
\sum_{\substack{s,t=1\\ s\ne t}}^k
\partial_{x_i x_{j_s}}^2f\,
\partial_{x_i x_{j_t}}^2f
\partial_{x_{j_s}}f\,\partial_{x_{j_t}}f
\prod_{\substack{\ell=1\\ \ell\ne s,t}}^k
\bigl(\partial_{x_{j_\ell}}f\bigr)^2
\\
&=
k(k-1)
\biggl(
\sum_{j=1}^n
(1-x_j^2)\partial_{x_j}f\,\partial_{x_i x_j}^2f
\biggr)^2
\biggl(
\sum_{m=1}^n
(1-x_m^2)(\partial_{x_m}f)^2
\biggr)^{k-2}
\\
&=
k(k-1)|\nabla f|_0^{2k-4}
\biggl(
\sum_{j=1}^n
(1-x_j^2)\partial_{x_j}f\,\partial_{x_i x_j}^2f
\biggr)^2.
\end{align*}
Therefore, for every fixed $i$,
\begin{align*}
&\sum_{j_1,\ldots,j_k=1}^n
\prod_{\ell=1}^k(1-x_{j_\ell}^2)
\bigl(\partial_{x_i}Q_{j_1,\ldots,j_k}\bigr)^2
\\
&=
k|\nabla f|_0^{2k-2}
\sum_{j=1}^n
(1-x_j^2)
\bigl(\partial_{x_i x_j}^2f\bigr)^2
+
k(k-1)|\nabla f|_0^{2k-4}
\biggl(
\sum_{j=1}^n
(1-x_j^2)\partial_{x_j}f\,\partial_{x_i x_j}^2f
\biggr)^2.
\end{align*}
Multiplying by $1-x_i^2$ and summing over $i=1,\ldots,n$,
we obtain
\begin{align*}
&\sum_{j_1,\ldots,j_k=1}^n
\prod_{\ell=1}^k(1-x_{j_\ell}^2)
\sum_{i=1}^n(1-x_i^2)
\bigl(\partial_{x_i}Q_{j_1,\ldots,j_k}\bigr)^2
\\
&=
k|\nabla f|_0^{2k-2}
\sum_{i,j=1}^n
(1-x_i^2)(1-x_j^2)
(\partial_{x_i x_j}^2f)^2
\\
&+
k(k-1)|\nabla f|_0^{2k-4}
\sum_{i=1}^n(1-x_i^2)
\biggl(
\sum_{j=1}^n
(1-x_j^2)\partial_{x_j}f\,\partial_{x_i x_j}^2f
\biggr)^2.
\end{align*}

We now estimate the term with
$\nabla|\nabla f|_0^2$. Since
\[
|\nabla f|_0^2
=
\sum_{j=1}^n(1-x_j^2)(\partial_{x_j}f)^2,
\]
we have
\[
\partial_{x_i}|\nabla f|_0^2
=
-2x_i(\partial_{x_i}f)^2
+
2\sum_{j=1}^n
(1-x_j^2)\partial_{x_j}f\,\partial_{x_ix_j}^2f .
\]
Hence, on the cube $[-1, 1]^n$,
\begin{align*}
\bigl|\nabla|\nabla f|_0^2\bigr|_0^2
&=
\sum_{i=1}^n(1-x_i^2)
\bigl(\partial_{x_i}|\nabla f|_0^2\bigr)^2
\\
&\le
8|\nabla f|_0^2|\nabla f|^2
+
8\sum_{i=1}^n(1-x_i^2)
\Bigl(
\sum_{j=1}^n
(1-x_j^2)\partial_{x_j}f\,\partial_{x_ix_j}^2f
\Bigr)^2.
\end{align*}
Multiplying by $|\nabla f|_0^{2k-4}$ and integrating, we get
\begin{align*}
&\int_{[-1,1]^n}
|\nabla f|_0^{2k-4}
\bigl|\nabla|\nabla f|_0^2\bigr|_0^2\,dx
\\
&\le
8\int_{[-1,1]^n}
|\nabla f|_0^{2k-2}|\nabla f|^2\,dx
+
8\int_{[-1,1]^n}
|\nabla f|_0^{2k-4}
\sum_{i=1}^n(1-x_i^2)
\Bigl(
\sum_{j=1}^n
(1-x_j^2)\partial_{x_j}f\,\partial_{x_ix_j}^2f
\Bigr)^2 dx.
\end{align*}
Therefore, applying \eqref{eq-MB-weighted}, we obtain
\begin{align*}
&\int_{[-1,1]^n}|\nabla f|_0^{2k-2}
\sum_{i,j=1}^n
(1-x_i^2)(1-x_j^2)
(\partial_{x_ix_j}^2f)^2\,dx
+
(k-1)^2
\int_{[-1,1]^n}
|\nabla f|_0^{2k-4}
\bigl|\nabla|\nabla f|_0^2\bigr|_0^2\,dx
\\
&\le
8\int_{[-1,1]^n}
\sum_{j_1,\ldots,j_k=1}^n
\prod_{\ell=1}^k(1-x_{j_\ell}^2)
\sum_{i=1}^n(1-x_i^2)
\bigl(\partial_{x_i}Q_{j_1,\ldots,j_k}\bigr)^2\, dx
\\
&\qquad\qquad\qquad\qquad\qquad\qquad\qquad\qquad\qquad\qquad\qquad+ 8(k-1)^2\int_{[-1,1]^n}
|\nabla f|_0^{2k-2}|\nabla f|^2\,dx
\\
&\le 
16k^2d^2\bigl\||\nabla f|_0\bigr\|_{L^{2k}([-1, 1]^n)}^{2k} + 8(k-1)^2
\bigl\||\nabla f|_0\bigr\|_{L^{2k}([-1, 1]^n)}^{2k-2}\|\nabla f\|_{L^{2k}([-1, 1]^n)}^{2}.
\end{align*}
Finally, by Lemma \ref{lem-2-2}
\[
\|\nabla f\|_{L^{2k}([-1,1]^n)}
\le
8kd\,
\||\nabla f|_0\|_{L^{2k}([-1,1]^n)},
\]
which implies the announced estimate since 
\[
16k^2d^2 + 512k^2d^2(k-1)^2\le 512k^4d^2.
\]
The lemma is proved.
\end{proof}

\subsection{Proof of Theorem~\ref{th-2}}

The case $k=1$ follows from Lemma~\ref{lem-2-4} with
$\alpha_1=\ldots=\alpha_n=0$:
\[
\bigl\||\nabla f|_0\bigr\|_{L^2([-1,1]^n)}
\le
\sqrt{d(d+1)}\,\|f\|_{L^2([-1,1]^n)}
\le \sqrt{2}d\,\|f\|_{L^2([-1,1]^n)}.
\]
By Lemma \ref{lem-2-2}, we also have
\[
\|\nabla f\|_{L^2([-1, 1]^n)}\le 8\sqrt{2}d^2\,\|f\|_{L^2([-1,1]^n)}\le 12d^2\,\|f\|_{L^2([-1,1]^n)}.
\]

Let now $k\ge2$. 
We first prove the estimate for the weighted gradient.

We apply Lemma~\ref{lem-2-5} with
\[
r=2k-2,
\quad
q=k-1.
\]
Then $r\ge q\ge1$, and we obtain
\begin{align*}
&\bigl\||\nabla f|_0\bigr\|_{L^{2k}([-1,1]^n)}^{2k}
\le
2\bigl\|f|\nabla f|_0^{k-1}\bigr\|_{L^2([-1, 1]^n)}
\biggl(
\bigl\||\nabla f|_0\bigr\|_{L^{2k}([-1,1]^n)}^{2k}
\\
&\qquad\qquad\qquad\qquad\qquad\qquad\qquad\qquad\qquad+
\int_{[-1,1]^n}|\nabla f|_0^{2k-2}
\sum_{j,k=1}^n
(1-x_j^2)(1-x_k^2)
(\partial_{x_jx_k}^2f)^2\,dx
\\
&\qquad\qquad\qquad\qquad\qquad\qquad\qquad\qquad\qquad+
(k-1)^2\int_{[-1,1]^n}
|\nabla f|_0^{2k-4}
\bigl|\nabla |\nabla f|_0^2\bigr|_0^2
\,dx
\biggr)^{1/2}.
\end{align*}
By Lemma~\ref{lem-2-6},
\begin{align*}
&\int_{[-1,1]^n}|\nabla f|_0^{2k-2}
\sum_{i,j=1}^n
(1-x_i^2)(1-x_j^2)
(\partial_{x_ix_j}^2f)^2\,dx
+
(k-1)^2
\int_{[-1,1]^n}
|\nabla f|_0^{2k-4}
\bigl|\nabla|\nabla f|_0^2\bigr|_0^2
\,dx
\\
&\le
512k^4d^2
\bigl\||\nabla f|_0\bigr\|_{L^{2k}([-1,1]^n)}^{2k}.
\end{align*}
Therefore,
\[
\int_{[-1,1]^n}|\nabla f|_0^{2k}\,dx
\le
2\bigl\|f|\nabla f|_0^{k-1}\bigr\|_{L^2([-1,1]^n)}
\Bigl(
(1+512k^4d^2)\bigl\||\nabla f|_0\bigr\|_{L^{2k}([-1,1]^n)}^{2k}
\Bigr)^{1/2}.
\]
Since $k\ge2$ and $d\ge1$,
\[
1+512k^4d^2\le 23^2k^4d^2.
\]
Hence
\[
\int_{[-1,1]^n}|\nabla f|_0^{2k}\,dx
\le
46k^2d\,
\bigl\|f|\nabla f|_0^{k-1}\bigr\|_{L^2([-1,1]^n)}
\bigl\||\nabla f|_0\bigr\|_{L^{2k}([-1,1]^n)}^k.
\]
By H\"older's inequality,
\[
\bigl\|f|\nabla f|_0^{k-1}\bigr\|_{L^2([-1,1]^n)}
\le
\|f\|_{L^{2k}([-1,1]^n)}
\bigl\||\nabla f|_0\bigr\|_{L^{2k}([-1,1]^n)}^{k-1}.
\]
Thus,
\[
\bigl\||\nabla f|_0\bigr\|_{L^{2k}([-1,1]^n)}^{2k}
\le
46k^2d\,
\|f\|_{L^{2k}([-1,1]^n)}
\bigl\||\nabla f|_0\bigr\|_{L^{2k}([-1,1]^n)}^{2k-1},
\]
which implies
\[
\bigl\||\nabla f|_0\bigr\|_{L^{2k}([-1,1]^n)}
\le
46k^2d\,\|f\|_{L^{2k}([-1,1]^n)}.
\]
Finally, Lemma~\ref{lem-2-2} gives
\[
\|\nabla f\|_{L^{2k}([-1,1]^n)}
\le
8kd\,
\bigl\||\nabla f|_0\bigr\|_{L^{2k}([-1,1]^n)}
\le
368k^3d^2\,\|f\|_{L^{2k}([-1,1]^n)}.
\]
This concludes the proof.
\qed

\section{Unimodal densities}

\begin{lemma}\label{lem-3-1}
Let $n, d\in \mathbb{N}$ and $p\ge 1$. Assume that
\begin{equation}\label{eq-cube-MB}
\|\nabla g\|_{L^p([-1,1]^n)}
\le
C(p)d^2\|g\|_{L^p([-1,1]^n)}
\quad \forall g\in\mathcal P_d(\mathbb R^n).
\end{equation}
Then, for every $a_j<b_j$, $j=1,\ldots,n$, and every
$f\in\mathcal P_d(\mathbb R^n)$,
\[
\int_{[a_1,b_1]\times\ldots\times[a_n,b_n]}
\Bigl(
\sum_{j=1}^n
(b_j-a_j)^2
\bigl(\partial_{x_j} f(x)\bigr)^2
\Bigr)^{p/2}
\,dx
\le
\bigl(2C(p)d^2\bigr)^p
\int_{[a_1,b_1]\times\ldots\times[a_n,b_n]} |f(x)|^p\,dx.
\]
\end{lemma}

\begin{proof}
Let
\[
L(y)
=
\Bigl(
\frac{b_1-a_1}{2}y_1+\frac{a_1+b_1}{2},
\ldots,
\frac{b_n-a_n}{2}y_n+\frac{a_n+b_n}{2}
\Bigr)
\]
and set
\[
g(y):=f(L(y)).
\]
Then $g\in\mathcal P_d(\mathbb R^n)$ and
\[
\partial_{y_j}g(y)
=
\frac{b_j-a_j}{2}\,
\partial_{x_j}f(L(y)).
\]
Hence
\[
|\nabla g(y)|^2
=
\frac14
\sum_{j=1}^n
(b_j-a_j)^2
\bigl(\partial_{x_j}f(L(y))\bigr)^2.
\]
Applying \eqref{eq-cube-MB} to $g$, we get
\[
\int_{[-1,1]^n}
\Bigl(
\frac14
\sum_{j=1}^n
(b_j-a_j)^2
\bigl(\partial_{x_j}f(L(y))\bigr)^2
\Bigr)^{p/2}
\,dy
\le
\bigl(C(p) d^2\bigr)^p
\int_{[-1,1]^n}|f(L(y))|^p\,dy.
\]
Changing variables $x=L(y)$ on both sides gives
the claimed estimate.
\end{proof}

\begin{lemma}\label{lem-3-2}
Let $n\in\mathbb N$, let $\varrho_1,\ldots,\varrho_n$ be bounded probability
densities on $\mathbb R$, and set
\[
M:=\max_{1\le j\le n}\|\varrho_j\|_\infty.
\]
For $t_j\ge0$, set
\[
E_j(t_j):=\{s\in\mathbb R\colon \varrho_j(s)\ge t_j\},
\quad
L_j(t_j):=\lambda(E_j(t_j)),
\]
where $\lambda$ is the Lebesgue measure on $\mathbb R$.
Then, for every $p\ge2$ and every Borel vector field
$\eta=(\eta_1,\ldots,\eta_n)\colon\mathbb R^n\to[0,+\infty)^n$, one has
\begin{equation}\label{eq-cake}
\int_{[0,\infty)^n}
\int_{E_1(t_1)\times\cdots\times E_n(t_n)}
\Bigl(\sum_{j=1}^n L_j(t_j)^2 \eta_j(x)\Bigr)^{p/2}
\,dx\,dt
\ge
\frac{1}{M^p}
\int_{\mathbb R^n}
\Bigl(\sum_{j=1}^n\eta_j(x)\Bigr)^{p/2}
\prod_{i=1}^n \varrho_i(x_i)\,dx.
\end{equation}
\end{lemma}

\begin{proof}
For each $j$, let
\[
M_j:=\|\varrho_j\|_\infty.
\]
By Fubini's theorem,
\[
\int_0^{M_j} L_j(t_j)\,dt_j
=
\int_{\mathbb R}\varrho_j(s)\,ds
=
1.
\]
Since $L_j$ is nonincreasing on $[0,M]$, for every $0<\tau<M_j$ we have
\[
\frac1\tau\int_0^\tau L_j(t_j)\,dt_j
\ge
L(\tau)
\ge
\frac{1}{M_j-\tau}\int_\tau^{M_j} L(t_j)\,dt_j.
\]
Therefore, for $0<\tau<M_j$,
\begin{align*}
\frac1\tau\int_0^\tau L(t_j)\,dt_j
&\ge
\frac{\tau}{M_j}\cdot\frac1\tau\int_0^\tau L(t_j)\,dt_j
+
\frac{M_j-\tau}{M_j}\cdot\frac{1}{M_j-\tau}\int_\tau^{M_j} L(t_j)\,dt_j
\\
&=
\frac{1}{M_j}\int_0^{M_j} L(t_j)\,dt_j
=
\frac{1}{M_j}.
\end{align*}
Hence, by the Cauchy--Schwarz inequality,
for every $0<\tau\le M_j$, we have
\[
\int_0^\tau L_j(t_j)^2\,dt_j
\ge
\tau
\biggl(\frac1\tau\int_0^\tau L_j(t_j)\,dt_j\biggr)^2
\ge
\frac{\tau}{M_j^2}.
\]

By Tonelli's theorem, the left-hand side of \eqref{eq-cake} equals
\[
\int_{\mathbb R^n}
\int_0^{\varrho_1(x_1)}\ldots\int_0^{\varrho_n(x_n)}
\Bigl(\sum_{j=1}^n L_j(t_j)^2\eta_j(x)\Bigr)^{p/2}
\,dt\,dx.
\]
Fix $x\in\mathbb R^n$ and assume first that
$\varrho_j(x_j)>0$ for all $j\in\{1,\ldots,n\}$. Since $p/2\ge1$, Jensen's
inequality gives
\begin{align*}
&\int_0^{\varrho_1(x_1)}\ldots\int_0^{\varrho_n(x_n)}
\Bigl(\sum_{j=1}^n L_j(t_j)^2\eta_j(x)\Bigr)^{p/2}
\,dt_1\ldots dt_n
\\
&\ge
\prod_{i=1}^n\varrho_i(x_i)
\biggl(
\frac{1}{\prod_{i=1}^n\varrho_i(x_i)}
\int_0^{\varrho_1(x_1)}\ldots\int_0^{\varrho_n(x_n)}
\sum_{j=1}^n L_j(t_j)^2\eta_j(x)
\,dt_1\ldots dt_n
\biggr)^{p/2}
\\
&=
\prod_{i=1}^n\varrho_i(x_i)
\biggl(
\sum_{j=1}^n
\eta_j(x)
\frac{1}{\varrho_j(x_j)}
\int_0^{\varrho_j(x_j)}L_j(t_j)^2\,dt_j
\biggr)^{p/2}
\\
&\ge
\prod_{i=1}^n\varrho_i(x_i)
\Bigl(\sum_{j=1}^n M_j^{-2}\eta_j(x)\Bigr)^{p/2}
\ge
\frac{1}{M^p}
\prod_{i=1}^n\varrho_i(x_i)
\Bigl(\sum_{j=1}^n\eta_j(x)\Bigr)^{p/2}.
\end{align*}
If $\varrho_j(x_j)=0$ for at least one $j$, then both sides of the pointwise
estimate above vanish, and the estimate is trivial.
Integrating over $x$ gives \eqref{eq-cake}.
\end{proof}

\begin{theorem}\label{th-3}
Let $n,d\in \mathbb N$ and $p\ge2$. Assume that
\[
\|\nabla g\|_{L^p([-1,1]^n)}
\le
C(p)d^2\|g\|_{L^p([-1,1]^n)}
\quad \forall g\in\mathcal P_d(\mathbb R^n).
\]
Let
$\mu=\mu_1\otimes\cdots\otimes\mu_n$,
where each $\mu_j$ has a bounded unimodal density $\varrho_j$ on $\mathbb R$.
Set
\[
M:=\max_{1\le j\le n}\|\varrho_j\|_\infty.
\]
Then, for every $f\in\mathcal P_d(\mathbb R^n)$, one has
\[
\|\nabla f\|_{L^p(\mu)}
\le
2C(p)Md^2
\|f\|_{L^p(\mu)}.
\]
\end{theorem}

\begin{proof}
For $t_j\ge0$, set
\[
E_j(t_j):=\{x\in\mathbb R\colon \varrho_j(x)\ge t_j\},
\quad
L_j(t_j):=\lambda(E_j(t_j)).
\]
Since $\varrho_j$ is unimodal, each nonempty set $E_j(t_j)$ is an interval.
Moreover, for $t_j>0$ it has finite length, and its endpoints do not affect
the integrals below. Therefore, applying Lemma~\ref{lem-3-1} to the rectangle
\[
E_1(t_1)\times\ldots\times E_n(t_n),
\]
for $t_1,\ldots,t_n>0$, we obtain
\[
\int_{E_1(t_1)\times\cdots\times E_n(t_n)}
\Bigl(
\sum_{j=1}^n
L_j(t_j)^2
\bigl(\partial_{x_j}f(x)\bigr)^2
\Bigr)^{p/2}
\,dx
\le
(2C(p)d^2)^p
\int_{E_1(t_1)\times\cdots\times E_n(t_n)}
|f(x)|^p\,dx.
\]
Integrating this inequality over $t=(t_1,\ldots,t_n)\in[0,\infty)^n$ and
using Tonelli's theorem, the right-hand side becomes
\[
(2C(p)d^2)^p
\int_{\mathbb R^n}|f(x)|^p
\prod_{j=1}^n\varrho_j(x_j)\,dx
=
(2C(p)d^2)^p
\|f\|_{L^p(\mu)}^p.
\]
By Lemma~\ref{lem-3-2}, applied to
\[
\eta_j(x):=\bigl(\partial_{x_j}f(x)\bigr)^2,
\]
the left-hand side is bounded from below by
\[
\frac{1}{M^p}
\int_{\mathbb R^n}
|\nabla f(x)|^p\,d\mu(x)
=
\frac{1}{M^p}\|\nabla f\|_{L^p(\mu)}^p.
\]
Combining the two bounds gives
\[
\frac{1}{M^p}\|\nabla f\|_{L^p(\mu)}^p
\le
(2C(p)d^2)^p\|f\|_{L^p(\mu)}^p.
\]
Taking the power $1/p$ gives the claimed estimate.
\end{proof}

Combining Theorem~\ref{th-3} and Theorem~\ref{th-2}
proves Theorem~\ref{th-4}.

\section{Freud-type densities}

\subsection{Integration by parts with respect to Freud measures}

For $m\in\mathbb{N}$ and $f\in C^\infty_P(\mathbb R^n)$, set
\[
L_m f
:=
\Delta f-2m\sum_{j=1}^n x_j^{2m-1}\partial_{x_j}f.
\]
For $j\in\{1,\ldots,n\}$, define
\[
\delta_{m,j}f
:=
2m x_j^{2m-1}f-\partial_{x_j}f.
\]
For a vector field $u\in C^\infty_P(\mathbb R^n,\mathbb R^n)$, define
\[
\delta_m u
:=
\sum_{j=1}^n\delta_{m,j}u_j
=
2m\sum_{j=1}^n x_j^{2m-1}u_j-\operatorname{div}u.
\]
Then
\[
\delta_m\nabla f=-L_m f.
\]

For $f,g\in C^\infty_P(\mathbb R^n)$, integration by parts gives
\[
\int_{\mathbb R^n}\partial_{x_j}f\,g\,d\nu_m^n
=
\int_{\mathbb R^n} f\,\delta_{m,j}g\,d\nu_m^n.
\]
Consequently, for $f\in C^\infty_P(\mathbb R^n)$ and
$u\in C^\infty_P(\mathbb R^n,\mathbb R^n)$,
\begin{equation}\label{eq-freud-ibp-1}
\int_{\mathbb R^n}\langle \nabla f,u\rangle\,d\nu_m^n
=
\int_{\mathbb R^n} f\,\delta_m u\,d\nu_m^n.
\end{equation}

We will also use the following analogue of the Gaussian identity
\eqref{int-by-parts-2}. Since
\[
\partial_{x_j}(\delta_m u)
=
\sum_{k=1}^n\delta_{m,k}\partial_{x_j}u_k
+
2m(2m-1)x_j^{2m-2}u_j,
\]
we obtain
\begin{align}
\label{eq-freud-ibp-2}
\int_{\mathbb R^n}(\delta_m u)^2\,d\nu_m^n
&=
\int_{\mathbb R^n}
\sum_{j=1}^n u_j\partial_{x_j}(\delta_m u)\,d\nu_m^n
\\
&=
2m(2m-1)
\int_{\mathbb R^n}
\sum_{j=1}^n x_j^{2m-2}u_j^2\,d\nu_m^n
+
\int_{\mathbb R^n}
\sum_{j,k=1}^n
\partial_{x_k}u_j\,\partial_{x_j}u_k\,d\nu_m^n
\notag
\\
&\le
2m(2m-1)
\int_{\mathbb R^n}
\sum_{j=1}^n x_j^{2m-2}u_j^2\,d\nu_m^n
+
\int_{\mathbb R^n}
\sum_{j,k=1}^n
(\partial_{x_k}u_j)^2\,d\nu_m^n.
\notag
\end{align}

\subsection{The $L^2$ Markov--Bernstein inequality}

\begin{theorem}\label{th-5}
Let $m\in\mathbb N$. There exists
a constant $C(m)>0$ such that, for every $n,d\in\mathbb N$ and every
$f\in\mathcal P_d(\mathbb R^n)$, one has
\[
\|\nabla f\|_{L^2(\nu_m^n)}
\le
C(m)d^{1-\frac1{2m}}
\|f\|_{L^2(\nu_m^n)}.
\]
\end{theorem}

\begin{proof}
We first prove a one-dimensional coefficient estimate. Let
$\{u_j\}_{j\ge0}$ be the orthonormal polynomial system in $L^2(\nu_m)$, where
$u_j$ has degree $j$. For $0\le j<r$, set
\[
b_{j,r}:=
\int_{\mathbb R}u_r'u_j\,d\nu_m.
\]
Since $u_r'$ has degree at most $r-1$, we have
\[
u_r'=\sum_{j=0}^{r-1}b_{j,r}u_j.
\]
Integrating by parts gives
\[
\int_{\mathbb R}u_r'u_j\,d\nu_m
=
-\int_{\mathbb R}u_ru_j'\,d\nu_m
+
2m\int_{\mathbb R}u_r(t)u_j(t)t^{2m-1}\,\nu_m(dt).
\]
The first integral on the right-hand side is zero, since
$u_j'$ has degree at most $j-1<r$ and $u_r$ is orthogonal to all polynomials
of degree smaller than $r$. Hence
\[
b_{j,r}
=
2m\int_{\mathbb R}u_r(t)u_j(t)t^{2m-1}\,\nu_m(dt).
\]
Now $t^{2m-1}u_j(t)$ has degree at most $j+2m-1$. Therefore it is orthogonal
to $u_r$ whenever $r>j+2m-1$. Thus
\[
b_{j,r}=0
\quad
\text{unless}
\quad
j<r<j+2m.
\]

Let $f\in\mathcal P_d(\mathbb R)$ and write
\[
f=\sum_{r=0}^d a_ru_r.
\]
Then
\[
f'
=
\sum_{r=1}^d a_ru_r'
=
\sum_{j=0}^{d-1}
\Bigl(
\sum_{\substack{r\in\{1,\ldots,d\}\\ j<r<j+2m}}
a_rb_{j,r}
\Bigr)u_j.
\]
By orthonormality and the Cauchy--Schwarz inequality,
\begin{align*}
\|f'\|_{L^2(\nu_m)}^2
&=
\sum_{j=0}^{d-1}
\Bigl(
\sum_{\substack{r\in\{1,\ldots,d\}\\ j<r<j+2m}}
a_rb_{j,r}
\Bigr)^2
\le
(2m-1)
\sum_{j=0}^{d-1}
\sum_{\substack{r\in\{1,\ldots,d\}\\ j<r<j+2m}}
a_r^2 b_{j,r}^2
\\
&\le
(2m-1)
\sum_{r=1}^d a_r^2\sum_{j=0}^{r-1} b_{j,r}^2
=
(2m-1)
\sum_{r=1}^d a_r^2\|u_r'\|_{L^2(\nu_m)}^2.
\end{align*}
By the one-dimensional Markov--Bernstein inequality for Freud weights \eqref{Freud-MB}, applied
to the polynomial $u_r$, we have
\[
\|u_r'\|_{L^2(\nu_m)}
\le
C_1(m)r^{1-\frac1{2m}}\|u_r\|_{L^2(\nu_m)}
=
C_1(m)r^{1-\frac1{2m}}.
\]
Consequently,
\begin{equation}\label{eq-1-dim-MB-Freud-coeff}
\|f'\|_{L^2(\nu_m)}^2
\le
C_2(m)
\sum_{r=1}^d r^{2-\frac1m}a_r^2.
\end{equation}

We now pass to the product measure. For a multi-index
$\mathbf r=(r_1,\ldots,r_n)\in\mathbb N_0^n$, set
\[
|\mathbf r|:=r_1+\ldots+r_n
\]
and
\[
U_{\mathbf r}(x):=
u_{r_1}(x_1)\ldots u_{r_n}(x_n).
\]
The system $\{U_{\mathbf r}\colon |\mathbf r|\le d\}$ is an orthonormal basis of
$\mathcal P_d(\mathbb R^n)$ in $L^2(\nu_m^n)$. Hence every polynomial
$f\in\mathcal P_d(\mathbb R^n)$ can be written as
\[
f=\sum_{|\mathbf r|\le d}a_{\mathbf r}U_{\mathbf r}.
\]

Fix $j\in\{1,\ldots,n\}$. For
\[
\widehat{\mathbf r}
=
(r_1,\ldots,r_{j-1},r_{j+1},\ldots,r_n)
\quad
\text{and}
\quad
\widehat x
=
(x_1,\ldots,x_{j-1},x_{j+1},\ldots,x_n),
\]
write
\[
U_{\widehat{\mathbf r}}(\widehat x)
=
\prod_{\ell\ne j}u_{r_\ell}(x_\ell).
\]
Then
\[
f(x)
=
\sum_{|\widehat{\mathbf r}|\le d}
U_{\widehat{\mathbf r}}(\widehat x)
f_{\widehat{\mathbf r}}(x_j),
\]
where
\[
f_{\widehat{\mathbf r}}(x_j)
:=
\sum_{r=0}^{d-|\widehat{\mathbf r}|}
a_{(\widehat{\mathbf r},r)}u_r(x_j).
\]
Here $a_{(\widehat{\mathbf r},r)}$ denotes the coefficient corresponding to
the multi-index whose $j$-th coordinate is $r$ and whose remaining coordinates
are given by $\widehat{\mathbf r}$. Therefore,
\[
\partial_{x_j}f(x)
=
\sum_{|\widehat{\mathbf r}|\le d}
U_{\widehat{\mathbf r}}(\widehat x)
f_{\widehat{\mathbf r}}'(x_j).
\]
Using the orthonormality of the functions $U_{\widehat{\mathbf r}}$ in
$L^2(\nu_m^{n-1})$, we obtain
\[
\int_{\mathbb R^n}
|\partial_{x_j}f|^2\,d\nu_m^n
=
\sum_{|\widehat{\mathbf r}|\le d}
\int_{\mathbb R}|f_{\widehat{\mathbf r}}'|^2\,d\nu_m.
\]
Applying \eqref{eq-1-dim-MB-Freud-coeff} to each
$f_{\widehat{\mathbf r}}$, we get
\[
\int_{\mathbb R^n}
|\partial_{x_j}f|^2\,d\nu_m^n
\le
C_2(m)
\sum_{|\widehat{\mathbf r}|\le d}
\sum_{r=0}^{d-|\widehat{\mathbf r}|}
r^{2-\frac1m}|a_{(\widehat{\mathbf r},r)}|^2
=
C_2(m)
\sum_{|\mathbf r|\le d}
r_j^{2-\frac1m}|a_{\mathbf r}|^2.
\]
Summing over $j\in\{1,\ldots,n\}$, we obtain
\[
\int_{\mathbb R^n}|\nabla f|^2\,d\nu_m^n
=
\sum_{j=1}^n
\int_{\mathbb R^n}|\partial_{x_j}f|^2\,d\nu_m^n
\le
C_2(m)
\sum_{|\mathbf r|\le d}
\Bigl(\sum_{j=1}^n r_j^{2-\frac1m}\Bigr)|a_{\mathbf r}|^2.
\]
Since $2-1/m\ge1$, we have
\[
\sum_{j=1}^n r_j^{2-\frac1m}
\le
\Bigl(\sum_{j=1}^n r_j\Bigr)^{2-\frac1m}
=
|\mathbf r|^{2-\frac1m}.
\]
Therefore,
\[
\int_{\mathbb R^n}|\nabla f|^2\,d\nu_m^n
\le
C_2(m)
\sum_{|\mathbf r|\le d}
|\mathbf r|^{2-\frac1m}|a_{\mathbf r}|^2
\le
C_2(m)d^{2-\frac1m}
\sum_{|\mathbf r|\le d}|a_{\mathbf r}|^2
=
C_2(m)d^{2-\frac1m}
\int_{\mathbb R^n}|f|^2\,d\nu_m^n.
\]
Taking square roots proves the theorem.
\end{proof}

\subsection{Key lemmas}

\begin{lemma}\label{lem-4-1}
Let $m\in\mathbb N$ and let $r\ge q>0$. There exists a constant
$C:=C(m,r,q)>0$ such that, for every $n,d\in\mathbb N$ and every
$f\in\mathcal P_d(\mathbb R^n)$, one has
\begin{align*}
&\|\nabla f\|_{L^{r+2}(\nu_m^n)}^{r+2}
\\
&\le
C
\|f|\nabla f|^{r-q}\|_{L^2(\nu_m^n)}
\biggl(
\int_{\mathbb R^n}
|\nabla f|^{2q}
\sum_{j=1}^n x_j^{2m-2}(\partial_{x_j}f)^2\,d\nu_m^n
+
\int_{\mathbb R^n}
|\nabla f|^{2q}\|D^2f\|_{\rm HS}^2\,d\nu_m^n
\biggr)^{1/2}.
\end{align*}
\end{lemma}

\begin{proof}The proof is the same as in the Gaussian case, with
\eqref{eq-freud-ibp-2} replacing \eqref{int-by-parts-2}.
Let $\eta\in C^\infty(\mathbb R)$ satisfy
\[
\eta(t)=0 \quad \text{for } |t|\le 1,
\quad
\eta(t)=1 \quad \text{for } |t|\ge 2,
\quad
0\le \eta(t)\le 1  \quad \text{for all } t\in\mathbb R.
\]
For $\varepsilon>0$, set $\eta_\varepsilon(t):=\eta(t/\varepsilon)$. For
$a>0$, define
\[
u_{a,\varepsilon}
:=
|\nabla f|^a\eta_\varepsilon(|\nabla f|)\nabla f.
\]
The expression is understood as zero on the set
$\{|\nabla f|\le \varepsilon\}$. With this convention, we have
$u_{a,\varepsilon}\in C^\infty_P(\mathbb R^n,\mathbb R^n)$.

A direct computation gives
\[
\delta_m u_{a,\varepsilon}
=
-|\nabla f|^a\eta_\varepsilon(|\nabla f|)L_mf
-
\bigl(
a|\nabla f|^{a-2}\eta_\varepsilon(|\nabla f|)
+
|\nabla f|^{a-1}\eta_\varepsilon'(|\nabla f|)
\bigr)
\langle D^2f\nabla f,\nabla f\rangle .
\]
Therefore,
\[
\delta_m u_{r,\varepsilon}
=
|\nabla f|^{r-q}
\bigl(
\delta_m u_{q,\varepsilon}
-
(r-q)|\nabla f|^{q-2}\eta_\varepsilon(|\nabla f|)
\langle D^2f\nabla f,\nabla f\rangle
\bigr).
\]

Using the integration by parts formula \eqref{eq-freud-ibp-1}, we obtain
\[
\int_{\mathbb R^n}
|\nabla f|^{r+2}I_{\{|\nabla f|\ge 2\varepsilon\}}\,
d\nu_m^n
\le
\int_{\mathbb R^n}
|\nabla f|^{r+2}\eta_\varepsilon(|\nabla f|)\,
d\nu_m^n
=
\int_{\mathbb R^n}
f\,\delta_m u_{r,\varepsilon}\,d\nu_m^n.
\]
Hence, using
\[
\langle D^2f\nabla f,\nabla f\rangle^2
\le
|\nabla f|^4\|D^2f\|_{\rm HS}^2,
\]
and applying the Cauchy--Schwarz inequality, we obtain
\begin{align*}
&\int_{\mathbb R^n}
|\nabla f|^{r+2}I_{\{|\nabla f|\ge 2\varepsilon\}}\,
d\nu_m^n
\\
&\le
C_1(r,q)
\|f|\nabla f|^{r-q}\|_{L^2(\nu_m^n)}
\biggl(
\int_{\mathbb R^n}
(\delta_m u_{q,\varepsilon})^2\,d\nu_m^n
+
\int_{\mathbb R^n}
|\nabla f|^{2q}
\|D^2f\|_{\rm HS}^2\,
d\nu_m^n
\biggr)^{1/2}.
\end{align*}

By
\eqref{eq-freud-ibp-2},
\begin{align*}
\int_{\mathbb R^n}
(\delta_m u_{q,\varepsilon})^2\,d\nu_m^n
&\le
C_2(m)
\int_{\mathbb R^n}
|\nabla f|^{2q}\sum_{j=1}^n x_j^{2m-2}(\partial_{x_j} f)^2\,
d\nu_m^n
+
\int_{\mathbb R^n}
\sum_{j,k=1}^n
\bigl(\partial_{x_k}(u_{q,\varepsilon})_j\bigr)^2\,
d\nu_m^n.
\end{align*}

It remains to estimate the last term. As in the Gaussian case, 
\[
\partial_{x_k}(u_{q,\varepsilon})_j
=
|\nabla f|^q\eta_\varepsilon(|\nabla f|)
\partial_{x_kx_j}^2f
+
\bigl(q|\nabla f|^{q-2}\eta_\varepsilon(|\nabla f|)
+
|\nabla f|^{q-1}\eta_\varepsilon'(|\nabla f|)\bigr)
\langle \nabla\partial_{x_k}f,\nabla f\rangle
\partial_{x_j}f
\]
and
\[
\sum_{j,k=1}^n
\bigl(\partial_{x_k}(u_{q,\varepsilon})_j\bigr)^2
\le
C(q)
\bigl(
1+
\|\eta'\|_\infty^2 I_{\{\varepsilon\le |\nabla f|\le 2\varepsilon\}}
\bigr)
|\nabla f|^{2q}\|D^2f\|_{\rm HS}^2.
\]
Combining the preceding estimates and passing to the limit as
$\varepsilon\to0$, with Lebesgue's dominated convergence theorem applied as in
the Gaussian case, gives the desired estimate.
\end{proof}

Applying the preceding lemma with
$r=2k-2$ and
$q=k-1$,
we obtain the following estimate.

\begin{corollary}\label{cor-4-1}
Let $m,k\in\mathbb N$, $k\ge2$. There exists a constant
$C=C(k,m)>0$ such that, for every $n,d\in\mathbb N$ and every
$f\in\mathcal P_d(\mathbb R^n)$, one has
\begin{align*}
&\|\nabla f\|_{L^{2k}(\nu_m^n)}^{2k}
\\
&\le
C\|f|\nabla f|^{k-1}\|_{L^2(\nu_m^n)}
\biggl(
\int_{\mathbb R^n}
|\nabla f|^{2k-2}
\sum_{j=1}^n x_j^{2m-2}(\partial_{x_j}f)^2\,d\nu_m^n
+
\int_{\mathbb R^n}
|\nabla f|^{2k-2}\|D^2f\|_{\rm HS}^2\,d\nu_m^n
\biggr)^{1/2}.
\end{align*}
\end{corollary}

\begin{lemma}\label{lem-4-2}
Let $m\in\mathbb N$. There exists a constant $C(m)>0$ such that, for every
$n, D\in\mathbb N$, every $j\in\{1,\ldots,n\}$, and every
$f\in\mathcal P_D(\mathbb R^n)$, one has
\[
\int_{\mathbb R^n}x_j^{2m-2}|f(x)|^2\,d\nu_m^n
\le
C(m)D^{1-\frac1m}
\|f\|_{L^2(\nu_m^n)}^2.
\]
\end{lemma}

\begin{proof}
By Fubini's theorem, it is sufficient to prove the one-dimensional estimate
\begin{equation}\label{eq-one-dim-mult-needed}
\int_{\mathbb R}t^{2m-2}|g(t)|^2\,\nu_m(dt)
\le
C(m)D^{1-\frac1m}
\|g\|_{L^2(\nu_m)}^2
\end{equation}
for every $g\in\mathcal P_D(\mathbb R)$.

For $m=1$, this is immediate. Assume that $m\ge2$. If
$\|g\|_{L^2(\nu_m)}=0$, there is nothing to prove. Otherwise, since
\[
\frac{d}{dt}\bigl(tg(t)^2e^{-|t|^{2m}}\bigr)
=
g(t)^2e^{-|t|^{2m}}
+
2tg(t)g'(t)e^{-|t|^{2m}}
-
2m t^{2m}g(t)^2e^{-|t|^{2m}},
\]
integration over $\mathbb R$ gives
\[
2m\int_{\mathbb R}t^{2m}|g(t)|^2\,\nu_m(dt)
=
\|g\|_{L^2(\nu_m)}^2
+
2\int_{\mathbb R}tg(t)g'(t)\,\nu_m(dt).
\]
By H\"older's inequality with exponents
$2m$, $\frac{2m}{m-1}$, and $2$,
\[
\left|\int_{\mathbb R}tg(t)g'(t)\,\nu_m(dt)\right|
\le
\Bigl(\int_{\mathbb R}t^{2m}|g(t)|^2\,\nu_m(dt)\Bigr)^{\frac1{2m}}
\|g\|_{L^2(\nu_m)}^{1-\frac1m}
\|g'\|_{L^2(\nu_m)}.
\]
By the one-dimensional $L^2$ Markov--Bernstein inequality for Freud weights \eqref{Freud-MB},
\[
\|g'\|_{L^2(\nu_m)}
\le
C_1(m)D^{1-\frac1{2m}}\|g\|_{L^2(\nu_m)}.
\]
Therefore,
\[
\int_{\mathbb R}t^{2m}|g(t)|^2\,\nu_m(dt)
\le
\|g\|_{L^2(\nu_m)}^2
+
C_1(m)D^{1-\frac1{2m}}
\Bigl(\int_{\mathbb R}t^{2m}|g(t)|^2\,\nu_m(dt)\Bigr)^{\frac1{2m}}
\|g\|_{L^2(\nu_m)}^{2-\frac1m}.
\]
Set
\[
a:=
\frac{\int_{\mathbb R}t^{2m}|g(t)|^2\,\nu_m(dt)}
{\|g\|_{L^2(\nu_m)}^2}.
\]
Then
\[
a
\le
1+
C_1(m)D^{1-\frac1{2m}}a^{\frac1{2m}}.
\]
If $a\ge1$, then, since $D\ge1$,
\[
a
\le
1+
C_1(m)D^{1-\frac1{2m}}a^{\frac1{2m}}
\le
(1+C_1(m))D^{1-\frac1{2m}}a^{\frac1{2m}}.
\]
Therefore,
\[
a\le C_2(m)D.
\]
If $a<1$, then
\[
a<1\le (1+C_2(m))D= C_3(m)D.
\]
Thus, in all cases,
\[
a\le C_3(m)D.
\]
Hence
\[
\int_{\mathbb R}t^{2m}|g(t)|^2\,\nu_m(dt)
\le
C_3(m)D\|g\|_{L^2(\nu_m)}^2.
\]
Finally, by H\"older's inequality with exponents $\frac{m}{m-1}$ and $m$,
\[
\int_{\mathbb R}t^{2m-2}|g(t)|^2\,\nu_m(dt)
\le
\Bigl(\int_{\mathbb R}t^{2m}|g(t)|^2\,\nu_m(dt)\Bigr)^{1-\frac1m}
\|g\|_{L^2(\nu_m)}^{\frac2m}
\le
C_4(m)D^{1-\frac1m}\|g\|_{L^2(\nu_m)}^2.
\]
This proves \eqref{eq-one-dim-mult-needed}, and the product estimate follows
by Fubini's theorem.
\end{proof}

\subsection{Proof of Theorem~\ref{th-6}}

The case $k=1$ is Theorem~\ref{th-5}. Assume now that $k\ge2$.
By Corollary~\ref{cor-4-1}, we have
\begin{align}\label{eq-cor-2}
&\|\nabla f\|_{L^{2k}(\nu_m^n)}^{2k}
\le
C_1(k, m)\|f|\nabla f|^{k-1}\|_{L^2(\nu_m^n)}
\biggl(
\int_{\mathbb R^n}
|\nabla f|^{2k-2}
\sum_{j=1}^n x_j^{2m-2}(\partial_{x_j}f)^2\,d\nu_m^n
\\
&\qquad\qquad\qquad\qquad\qquad\qquad\qquad\qquad\qquad\qquad\qquad+
\int_{\mathbb R^n}
|\nabla f|^{2k-2}\|D^2f\|_{\rm HS}^2\,d\nu_m^n
\biggr)^{1/2}.\notag
\end{align}
By H\"older's inequality,
\[
\|f|\nabla f|^{k-1}\|_{L^2(\nu_m^n)}
\le
\|f\|_{L^{2k}(\nu_m^n)}
\|\nabla f\|_{L^{2k}(\nu_m^n)}^{k-1}.
\]
For each multi-index $(j_1,\ldots,j_k)\in\{1,\ldots,n\}^k$, let
\[
Q_{j_1,\ldots,j_k}
:=
\partial_{x_{j_1}}f\ldots\partial_{x_{j_k}}f.
\]
Then $Q_{j_1,\ldots,j_k}\in\mathcal P_{kd}(\mathbb R^n)$. By
Lemma~\ref{lem-1-2},
\[
\sum_{1\le j_1,\ldots,j_k\le n}
|\nabla Q_{j_1,\ldots,j_k}|^2
\ge
k|\nabla f|^{2k-2}\|D^2 f\|_{\rm HS}^2.
\]
Applying Theorem~\ref{th-5} to each $Q_{j_1,\ldots,j_k}$ and summing over
all $(j_1,\ldots,j_k)$, we obtain
\begin{align*}
\int_{\mathbb R^n}
|\nabla f|^{2k-2}\|D^2f\|_{\rm HS}^2\,d\nu_m^n
&\le
\sum_{1\le j_1,\ldots,j_k\le n}
\int_{\mathbb R^n}|\nabla Q_{j_1,\ldots,j_k}|^2\,d\nu_m^n
\\
&\le
C_2(k,m)d^{2-\frac1m}
\sum_{1\le j_1,\ldots,j_k\le n}
\int_{\mathbb R^n}|Q_{j_1,\ldots,j_k}|^2\,d\nu_m^n
\\
&=
C_2(k,m)d^{2-\frac1m}
\|\nabla f\|_{L^{2k}(\nu_m^n)}^{2k}.
\end{align*}
Similarly,
\[
|\nabla f|^{2k-2}
\sum_{j=1}^n x_j^{2m-2}(\partial_{x_j}f)^2
=
\sum_{1\le j_1,\ldots,j_k\le n}
x_{j_k}^{2m-2}|Q_{j_1,\ldots,j_k}|^2.
\]
By Lemma~\ref{lem-4-2}, applied to each $Q_{j_1,\ldots,j_k}$, we conclude
\begin{align*}
\int_{\mathbb R^n}
|\nabla f|^{2k-2}
\sum_{j=1}^n x_j^{2m-2}(\partial_{x_j}f)^2\,d\nu_m^n
&=
\sum_{1\le j_1,\ldots, j_k\le n }
\int_{\mathbb R^n}x_{j_k}^{2m-2}|Q_{j_1, \ldots, j_k}|^2\,d\nu_m^n
\\
&\le 
C_3(k, m)d^{1-\frac{1}{m}}
\sum_{1\le j_1,\ldots, j_k\le n }
\int_{\mathbb R^n}|Q_{j_1, \ldots, j_k}|^2\,d\nu_m^n
\\
&=
C_3(k, m)d^{1-\frac{1}{m}} \|\nabla f\|_{L^{2k}(\nu_m^n)}^{2k}.
\end{align*}

Substituting these estimates into \eqref{eq-cor-2}, we
obtain
\[
\|\nabla f\|_{L^{2k}(\nu_m^n)}^{2k}
\le
C_4(k,m)d^{1-\frac1{2m}}
\|f\|_{L^{2k}(\nu_m^n)}
\|\nabla f\|_{L^{2k}(\nu_m^n)}^{2k-1}.
\]
If $\|\nabla f\|_{L^{2k}(\nu_m^n)}=0$, the result is trivial.
Otherwise, dividing by
$\|\nabla f\|_{L^{2k}(\nu_m^n)}^{2k-1}$
gives the claimed estimate.
\qed

\section*{Use of AI Tools}

ChatGPT was used for language editing, stylistic suggestions, draft wording for
selected passages, and help with locating some references. All AI-generated
text and suggested references were checked, corrected where necessary, and
substantially revised by the author. The author takes full responsibility for
the content of the paper.





\section*{Acknowledgements}

The author would like to thank Sergey Tikhonov for reading the manuscript and for valuable comments and suggestions.

\smallskip

The research was supported by the AEI grant 
RYC2023-043616-I 
and by the Spanish State Research Agency, through the Severo Ochoa and Mar\'ia de Maeztu Program for Centers and
Units of Excellence in R\&D (CEX2020-001084-M).
The author thanks CERCA Programme (Generalitat de Catalunya) for institutional support.
{\sloppy
	
}

\end{document}